\documentclass[11 pt]{amsart}
\usepackage{amsmath}
\usepackage{amsfonts}
\usepackage{enumerate}
\usepackage{mathrsfs}
\usepackage{verbatim}
\usepackage[pdftex]{graphicx}
\DeclareGraphicsExtensions{.png,}
\newtheorem{lemma}{Lemma}
\newtheorem{theorem}{Theorem}
\newtheorem{example}{Example}
\newtheorem{prop}{Proposition}

\newtheorem{defn}{Definition}
\numberwithin{equation}{section}
\numberwithin{defn}{section}
\numberwithin{prop}{section}
\numberwithin{lemma}{section}

\newcommand{\R}{\mathbb{R}}
\newcommand{\Prob}{\mathbb{P}}
\newcommand{\Z}{\mathcal{Z}}
\newcommand{\N}{\mathbb{N}}
\newcommand{\Q}{\mathbb{Q}}

\newcommand{\p}{\mathbf{p}}
\newcommand{\q}{\mathbf{q}}

\newcommand{\W}{\mathcal{W}}
\newcommand{\B}{\mathcal{B}}
\newcommand{\I}{\mathcal{I}}
\newcommand{\var}{\text{var}}

\newcommand{\D}{\mathcal{D}}
\newcommand{\E}{\mathcal{E}}

\newcommand{\A}{\mathcal{A}}

\newcommand{\M}{\mathcal{M}}
\newcommand{\F}{\mathcal{F}}
\newcommand{\diam}{\mbox{diam}}

\renewcommand{\dim}{\mbox{dim}_{\mathcal{H}}}

\title[Infinite Non-Conformal Iterated Function Systems]{Infinite Non-Conformal Iterated Function Systems}
\author{Henry WJ Reeve}
\address{Henry WJ Reeve\\ Department of Mathematics\\ The University of Bristol\\
University Walk\\Clifton\\ Bristol\\BS8 1TW\\UK}
\email{henrywjreeve@googlemail.com}
\begin{document}
\maketitle
\begin{abstract} We consider a generalisation of the self-affine iterated function systems of Lalley and Gatzouras by allowing for a countable infinity of non-conformal contractions. It is shown that the Hausdorff dimension of the limit set is equal to the supremum of the dimensions of compactly supported ergodic measures. In addition we consider the multifractal analysis for countable families of potentials. We obtain a conditional variational principle for the level sets.
\end{abstract}

\section{Introduction}
Suppose we have a compact metric space $X$ together with a finite or countable family $\F=\{S_{d}\}_{ d \in \D}$ of uniformly contracting maps $S_d: X \rightarrow X$. The attractor $\Lambda$ of $\F$ is given by,
\begin{eqnarray}
\Lambda:= \bigcup_{\omega \in \D^{\N}} \bigcap_{n \in \N} S_{\omega_1} \circ \cdots \circ S_{\omega_n}(X).
\end{eqnarray} 
When $\F$ consists of finitely many conformal contractions, satisfying the open set condition, the Hausdorff dimension of the attractor $\dim \Lambda$ is given by Bowen's formula as the unique zero of an associated pressure function \cite{Bowen's formula}, which is equal to the supremum over the dimensions of ergodic measures supported on the limit set (see \cite{Pesin} for details). When $\F$ consists of a countable infinity of conformal contractions, satisfying the open set condition, the associated pressure function may not pass through zero. Nonetheless, Mauldin and Urba\'{n}ski have shown that $\Lambda$ satisfies a modified version of Bowen's formula in which $\dim \Lambda$ is given by the infimum over all values for which the pressure function is negative \cite{Mauldin Urbanski}. Moreover, $\dim \Lambda$ is equal to the supremum over the dimensions of ergodic measures supported on compact invariant subsets of the limit set.

When $\Lambda$ is non-conformal much less is known. In \cite{Falconer Singular Value} Falconer showed that, when $\F$ consists of finitely many affine contractions, $\dim \Lambda$ is bounded above by the unique zero of an associated subadditive pressure function. Moreover, for typical $\F$, with respect to an appropriate parameterization, this value also gives a lower bound. 

However, there are very few cases in which the Hausdorff dimension of a particular non-conformal limit set is known. In most such cases the upper bound given by the subadditive pressure function is non-optimal. The first examples of this type were provided by Bedford \cite{Bedford} and McMullen \cite{McMullen}. These constructions were generalised to include a continuum of examples with variable Lyapunov exponent by Lalley and Gatzouras in \cite{Gatzouras Lalley}. A generalisation in a different direction was given by Bara\'{n}ski in \cite{Baranski}. In \cite{LuziaETDS} Luzia considers certain non-conformal and non-linear repellers closely related to the self-affine limit sets of Lalley and Gatzouras. In each of these cases $\dim \Lambda$ is equal to the supremum over the dimensions of ergodic measures supported on the limit set.

Having determined the Hausdorff dimension of the the limit set in cases where the contractions are both non-conformal and have a variable Lyapunov exponent, it is natural to consider examples of iterated function systems consisting of a countable infinity of non-conformal contractions.

\begin{example}[Gauss-R\'{e}nyi Products]\label{GR products} Given $x,y \in [0,1]$ and $n \in \N$ we let $a_n(x)\in \N$ denote the $n$th digit in the continued fraction expansion of $x$ and $b_n(y) \in \{0,1\}$ denote the $n$th digit in the binary expanion of $y$. Choose some digit set $\D \subseteq \N \times \{0,1\}$ and define,
\begin{equation*}
\Lambda:= \left\lbrace (x,y) \in [0,1]^2: (a_n(x), b_n(y)) \in \D \text{ for all }n \in \N \right\rbrace.
\end{equation*}
Then $\Lambda$ is the attractor of the iterated function system consisting of all maps of the form,
\begin{equation*}
(x,y) \mapsto \left( \frac{1}{a+x}, \frac{y+b}{2} \right) \text{ for }(x,y) \in [0,1]^2,
\end{equation*}
with $(a,b) \in \D$.
\end{example}

This example is a member of a class of constructions which we shall refer to as INC systems (see Section 2 for the definition). For all such systems we shall show that the Hausdorff dimension of the limit system is equal to the supremum over the diminsions of ergodic measures supported on compact invariant subsets of the limit set.

We shall also consider the multifractal analysis of Birkhoff averages. When $\F$ consists of finitely many conformal contractions the spectrum is well understood  \cite{Barriera Saussol, Fan Feng Wu, Olsen Multifractal 1, Olsen Winter, Pesin Weiss Birkhoff}; the dimension of the level sets is given by a conditional variational principle (see \cite[Chapter 9]{Barreira book} for details). For a useful survey on related multifractal results we recommend \cite{Cl}.

Recently there has also been a great deal of work dealing with cases in which $\F$ consists of a countable infinity of conformal contractions (see \cite{K1,K2, K3, Iommi Jordan, L1, L2, L3}). In this setting the Birkhoff spectra can display a wide variety of interesting behaviour. For example, due to the unbounded contraction rates one can have phase transitions and flat regions in the spectrum (see \cite{K2, Iommi Jordan}). In addition, when dealing with a countable infinity of potentials on an infinite IFS, one does not obtain the usual conditional variational principle (see \cite[Theorem 1.1]{L3} for example). This is a consequence of the lack of both compactness and upper semi-continuity of entropy for countable state systems. 

There has also been some work on the multifractal analysis of Birkhoff averages for $\F$ consisting of finitely many affine planar contractions with a diagonal linear part. In \cite{Jordan Simon} Jordan and Simon gave a conditional variational principle which holds for typical members of parameterizable families of examples. In \cite{Barral Mensi, Barral Feng} Barral, Mensi and Feng investigated the multifractal analysis of Birkhoff averages on the self-affine limit sets of Bedford and McMullen \cite{Bedford, McMullen}. In particular, they obtain a conditional variational principle for the level sets \cite{Barral Feng}. In \cite{My LG paper} this result is extended to include the self-affine limit sets of Lalley and Gatzouras. However, the method used in \cite{My LG paper} relies heavily upon the compactness of the associated symbolic space.

In this paper we shall consider the multifractal analysis of Birkhoff averages for a family of iterated function systems consisting of a countable infinity of non-conformal contractions which we shall refer to as INC systems. We shall obtain a conditional variational principle for the level sets of a countable infinity of Birkhoff averages on the limit set for an INC system. 

\section{Notation and statement of results}\label{Results section}

Let $I:=[0,1]$ denote the closed unit interval. Given a digit set $\B$ we let $\B^*:=\bigcup_{n \in \N}\B^n$ denote the space of all finite strings. Given a sequence of maps $\{f_{j}\}_{j \in \B}$ indexed by $\B$ and a finite string $\omega=(\omega_1,\cdots,\omega_n) \in \B^*$ we let $f_{\omega}$ denote the composition $f_{\omega}:=f_{\omega_1} \circ \cdots \circ f_{\omega_n}$.

\begin{defn}[Interval Iterated Function Systems]\label{IIFS} By an \textit{interval iterated function system} we shall mean a family $\{f_j:j \in \B\}$ of $C^1$ maps $f_j:I \rightarrow I$, indexed over some finite or countable digit set $\B$, which satisfies the following assumtions.
\begin{enumerate}
\item[(UCC)]Uniform Contraction Condition. There exists a contraction ratio $\xi \in (0,1)$ and $N \in \N$ such that for all $n\geq N$ and all $\omega \in \B^n$ we have \[\sup_{x \in I}|f'_{\omega}(x)|\leq \xi^n.\]  
\item[(OIC)]Open Interval Condition. For all $j_1, j_2 \in \B$ with $j_1\neq j_2$, we have \[f_{j_1}((0,1))\cap f_{j_2}((0,1))=\emptyset.\]
\item[(TDP)]Tempered Distortion Property. There exists some sequence $\rho_n$ with $\lim_{n\rightarrow \infty} \rho_n=0$ such that for all $n \in \N$ and for all $\omega\in \B^n$ and all $x, y \in I$ we have \[e^{-n\rho_n}\leq \frac{|f'_{\omega}(x)|}{|f'_{\omega}(y)|}\leq e^{n\rho_n}.\]
If $\B$ is finite then $\{f_j:j \in \B\}$ is said to be a \textit{finite interval iterated function system}.
\end{enumerate}
\end{defn}

\begin{defn}[INC Systems] Suppose we have a finite interval iterated function system $\{g_i: i\in \A\}$ and for each $i \in \A$ we have a (finite or countable) interval iterated function system $\{f_{ij}:j \in \B_i\}$ with $\sup_{x \in \I} |f'_{ij}(x)|\leq \inf_{x \in \I} |g'_i(x)|$ for each $j \in \B_i$. Let $\D:=\left\lbrace (i,j): i \in \A, j \in \B_i\right\rbrace$ and for each pair $(i,j) \in \D$ we let $S_{ij}$ denote the map given by
\[S_{ij}(x,y)= (f_{ij}(x),g_i(y))\text{ for }(x,y) \in I^2.\]
An iterated function system $\left\lbrace S_{ij}:(i,j) \in \D\right\rbrace$ defined in this way shall be referred to as an \textit{INC System}.
\end{defn}

We shall use the symbolic spaces $\Sigma:=\D^{\N}$, and $\Sigma_v:=\A^{\N}$, each of which is endowed with the product topology. Let $\sigma:\Sigma \rightarrow \Sigma$ and $\sigma_v:\Sigma_v \rightarrow \Sigma_v$ denote the corresponding left shift operators. We let $\pi: \Sigma \rightarrow \Sigma_v$ denote the projection given by $\pi(\omega)=(i_{\nu})_{\nu\in \N}$ for $\omega=((i_{\nu},j_{\nu}))_{\nu \in \N}\in \Sigma$. We also let $\pi(((i_{\nu},j_{\nu}))_{\nu=1}^n)=(i_{\nu})_{\nu =1}^n$ for a finite string $(i_{\nu},j_{\nu}))_{\nu=1}^n \in \D^n$.
We define a projection $\Pi: \Sigma \rightarrow I^2$ by
\begin{equation}
\Pi(\omega):=\lim_{n\rightarrow \infty}S_{\omega_1}\circ \cdots S_{\omega_n}\left(I^2\right)\text{   for   }\omega=(\omega_n)_{n \in \N}\in \Sigma.
\end{equation}
We also define a vertical projection $\Pi_v: \Sigma_v \rightarrow I$ by
\begin{eqnarray}
\Pi_v(\mathbf{i})&:=&\lim_{n\rightarrow \infty}g_{i_1}\circ \cdots g_{i_n}\left(I^2\right)\text{   for   } \mathbf{i}=(i_{\nu})_{\nu \in \N}\in \Sigma_v.
\end{eqnarray}

Let $\Lambda:= \Pi(\Sigma)$. It follows that,
\begin{equation}
\Lambda=\bigcup_{(i,j) \in \D} S_{ij}(\Lambda).
\end{equation}
Given any finite subset $\F \subset \D$ we let $\Lambda_{\F}$ denote the unique non-empty compact set satisfying,
\begin{equation}
\Lambda_{\F}=\bigcup_{(i,j) \in \F} S_{ij}(\Lambda).
\end{equation}
In addition we define $\chi \in \Sigma \rightarrow \R$ and $\psi \in \Sigma_v \rightarrow \R$ by
\begin{eqnarray}
\chi(\omega)&:=& -\log |f_{\omega_1}'\left(\Pi(\sigma \omega)\right)| \text{   for   } \omega=(\omega_{\nu})_{\nu \in \N}\in \Sigma,\\
\psi(\mathbf{i})&:=& -\log |g_{i_1}'\left(\Pi_v(\sigma \mathbf{i})\right)|\text{   for   } \mathbf{i}=(i_{\nu})_{\nu \in \N}\in \Sigma_v.
\end{eqnarray}

Let $\mathscr{A}$ denote the Borel sigma algebra on $\Sigma_v$.
We let $\M_{\sigma}(\Sigma)$ denote the set of all $\sigma$-invariant Borel probability measures on $\Sigma$ and let $\M_{\sigma}^*(\Sigma)$ denote the set of $ \mu \in \M_{\sigma}(\Sigma)$ which are supported on a compact subset of $\Sigma$. Similarly we let $\E_{\sigma}(\Sigma)$ denote the set of $ \mu \in \M_{\sigma}(\Sigma)$ which are ergodic and  $\E_{\sigma}^*(\Sigma)$  denote the set of $\mu \in \E_{\sigma}(\Sigma)$ which are compactly supported. Given $\mu \in \M^*_{\sigma}(\Sigma)$ we define
\begin{eqnarray}
D(\mu)&:=& \frac{h_{\mu}(\sigma|\pi^{-1}\mathscr{A})}{\int \chi d \mu }+\frac{ h_{\mu\circ \pi^{-1}}( \sigma_v)}{\int \psi d \mu\circ \pi^{-1}}.
\end{eqnarray}
By the Ledrappier and Young dimension formula (see \cite[Corollary D]{LYMED2}) $D(\mu)$ gives the dimension of $\mu$ for all $\mu \in \E_{\sigma}^*(\Sigma)$.

\begin{theorem}\label{countable attractor} Let $\Lambda$ be the attractor of an INC system. Then,
\begin{eqnarray*}
\dim \Lambda &=& \sup \left\lbrace D(\mu):\mu \in \E^*_{\sigma}(\Sigma) \right\rbrace\\
&=& \sup \left\lbrace D(\mu):\mu \in \M^*_{\sigma}(\Sigma) \right\rbrace\\
&=& \sup \left\lbrace \dim \Lambda_{\F}: \F \text{ is a finite subset of }\D \right\rbrace.\\
\end{eqnarray*}
\end{theorem}

Given a potential $\varphi:\Sigma \rightarrow \R$ and $n \in \N$ we shall let $S_n(\varphi):= \sum_{l=0}^{n-1}\varphi\circ \sigma^l$, $A_n(\varphi):=n^{-1}S_n(\varphi)$ and define
\begin{equation*}
\var_n(\varphi):=\sup \left\lbrace |\varphi(\omega)-\varphi(\tau)|:\omega, \tau \in \Sigma \text{ with } \omega_l=\tau_l \text{ for }l=1,\cdots, n\right\rbrace.
\end{equation*}

\begin{defn}[Tempered Distortion Property]\label{TDP} A potential $\varphi : \Sigma \rightarrow \R$ is said to satisfy the tempered distortion property if
$\lim_{n\rightarrow \infty} \var_n(A_n(\varphi))=0.$
\end{defn}
It follows from the tempered distortion property of $\{f_{ij}: j \in \B_i\}$ and  $\{g_i: i \in \A\}$ (see Definition \ref{IIFS} (TDP)) that both $\chi$ and $\psi\circ\pi$ satisfy the tempered distortion property in Definition \ref{TDP}. We shall focus on potentials satisfying the tempered distortion property which are bounded on one side. That is, there exists some $a \in \R$ such that either $ \varphi(\omega) \leq a$ for all $\omega \in \Sigma$ or $ \varphi(\omega) \geq a$ for all $\omega \in \Sigma$. Note that this family includes every positive valued uniformly continuous potential.

Suppose we have a countable family $(\varphi_k)_{k\in \N}$ of real-valued potentials $\varphi_k \Sigma \rightarrow \R$, together with some $(\alpha_k)_{k \in \N} \subset \R \cup \{-\infty, +\infty\}$. We define,
\begin{equation}
E_{\varphi}(\alpha):=\left\lbrace \omega \in \Sigma :\lim_{n\rightarrow \infty} A_n(\varphi_k)(\omega)=\alpha_k \text{ for all } k \in \N \right\rbrace,
\end{equation}                                  
and let $J_{\varphi}(\alpha):=\Pi( E_{\varphi}(\alpha))$. Here limits are taken with respect to the usual two point compactification of $\R$.

Given $\alpha \in \R\cup \{-\infty, +\infty\}$ we define a shrinking family $\{B_m(\alpha)\}_{m \in \N}$ of neighbourhoods of $\alpha$ by,
\begin{eqnarray}
B_m(\alpha):= \begin{cases} \left\lbrace x: |x - \alpha|< \frac{1}{m} \right\rbrace \text{ if } \alpha \in \R\\ (m , +\infty)\text{ if } \alpha = +\infty
\\ (-\infty, -m)\text{ if } \alpha = -\infty.
\end{cases}
\end{eqnarray}

\begin{theorem}\label{Main Countable} Suppose we have countably many potentials $(\varphi_k)_{k \in \N}$ each of which satisfies the tempered distortion property and is bounded on one side. Then, for all $\alpha=(\alpha_k)_{k \in \N}\in \left(\R\cup \{-\infty, + \infty\}\right) ^{\N}$ we have, 
\begin{eqnarray*}
\dim J_{\varphi}(\alpha) &=& \lim_{m \rightarrow \infty}\sup\left\lbrace D(\mu):  \mu \in \E_{\sigma}^*(\Sigma), \int \varphi_k d\mu \in B_{m}(\alpha_k)   \text{ for } k \leq m\right\rbrace\\ &=&
\lim_{m \rightarrow \infty}\sup\left\lbrace D(\mu):  \mu \in \M_{\sigma}^*(\Sigma), \int \varphi_k d\mu \in B_{m}(\alpha_k)   \text{ for } k \leq m\right\rbrace.
\end{eqnarray*}
\end{theorem}

Note that in general it is impossible to remove the dependence on $m$ and obtain a variational principle of the form \cite[Theorem 9.1.4]{Barreira book}. This is a consequence of lack of compactness in the symbolic space, along with the lack of upper semi-continuity for entropy. Example \ref{Escape of Mass Example} illustrates this phenomenon.

Nonetheless for the interior of the spectrum for a single potential we can use an argument from Iommi and Jordan \cite{Iommi Jordan} to recover the usual conditional variational principle. 
Let $\alpha_{\min}:= \inf \left\lbrace \int \varphi d\mu : \mu \in \M_{\sigma}(\Sigma)\right\rbrace$ and $\alpha_{\max}:= \sup \left\lbrace \int \varphi d\mu : \mu \in \M_{\sigma}(\Sigma)\right\rbrace$.

\begin{theorem}\label{one potential} Given a non-negative potential $\varphi: \Sigma \rightarrow \R$ satisfying the tempered distortion property and some $\alpha \in (\alpha_{\min}, \alpha_{\max})$ we have
\begin{equation*}
\dim J_{\varphi}(\alpha)=\sup \left\lbrace D(\mu):\mu \in \M^*_{\sigma}(\Sigma),   \int \varphi d \mu= \alpha  \right\rbrace.
\end{equation*}
\end{theorem}
\begin{proof} One can argue as in \cite[Lemma 3.2]{Iommi Jordan}, by taking convex combinations, to see that the supremum on the right depends continuously on $\alpha$. Consequently,
\begin{eqnarray*}
\lim_{m \rightarrow \infty}\sup \left\lbrace D(\mu):\mu \in \M^*_{\sigma}(\Sigma),   \int \varphi d \mu \in  B_m(\alpha)  \right\rbrace\\
\leq \sup \left\lbrace D(\mu):\mu \in \M^*_{\sigma}(\Sigma),   \int \varphi d \mu= \alpha  \right\rbrace.
\end{eqnarray*}
Thus, Theorem \ref{one potential} follows from Theorem \ref{Main Countable}.
 \end{proof}
The following examples are applications of Theorem \ref{Main Countable}.
\begin{example}[Geometric Arithmetic Mean Sets] Let $\Lambda$ be as in Example \ref{GR products}. For each $\alpha, \beta \in \R$ we define,
\begin{equation*}
\Lambda^{\times}(\alpha, \beta):= \left\lbrace (x,y) \in \Lambda: \lim_{n \rightarrow \infty}\sqrt[n]{ a_1(x) \cdots a_n(x) }= \alpha, \lim_{n \rightarrow \infty}\frac{b_1(y)+ \cdots + b_n(y)}{n}=\beta \right\rbrace.
\end{equation*}
Then $\dim \Lambda^{\times}(\alpha,\beta)$ varies continuously as a function of $(\alpha, \beta) \in \R^2$.
\end{example}

\begin{example}[Arithmetic Mean Sets] Let $\Lambda$ be as in Example \ref{GR products}. For each $\alpha, \beta \in \R$ we define,
\begin{equation*}
\Lambda^{+}(\alpha, \beta):= \left\lbrace (x,y) \in \Lambda: \lim_{n \rightarrow \infty}\frac{a_1(x)+ \cdots + a_n(x)}{n}=\alpha, \lim_{n \rightarrow \infty}\frac{b_1(y)+ \cdots + b_n(y)}{n}=\beta \right\rbrace.
\end{equation*}
Then $\dim \Lambda^{+}(\alpha,\beta)$ varies continuously as a function of $(\alpha, \beta) \in \R^2$.
\end{example}

\begin{example}[Total Escape of Mass] \label{Escape of Mass Example} Within the setting of Example \ref{GR products} we consider the set,
\begin{equation*}
\Lambda_{\infty}(\D):=\left\lbrace (x,y) \in \Lambda: \lim_{n \rightarrow \infty}\frac{\#\left\lbrace l \leq n: a_l(x)=m\right\rbrace}{n}=0\text{ for all } m \in \N \right\rbrace.
\end{equation*}
If $\D$ is finite then $\Lambda_{\infty}(\D)$ is clearly empty. However, if $\D:= \left\lbrace (n, n \mod 2): n \in \N\right \rbrace$, then $\dim \Lambda_{\infty}(\D)=\frac{3}{2}$.
\end{example}

The rest of the paper will be direceted towards proving Theorem \ref{Main Countable}, from which Theorems \ref{one potential} and \ref{countable attractor} follow. The proof will consist of an upper bound, contained in sections \ref{locally constant} and \ref{Proof UB} and a lower bound, contained in sections \ref{Preliminaries for lower bound} and \ref{Proof LB}. We begin the proof of the upper bound by proving an upper estimate, in Section \ref{locally constant}, for the dimension of the level sets in the special case in which we have have finitely many locally constant potentials. It is in proving this initial upper estimate that many of the difficulties lie. We use the compactness of the vertical symbolic space $\Sigma_v=\A^{\N}$ to partition the symbolic level sets into a countable number of sets for which certain sequences depending only upon $\pi(\omega)$ converge to some prescribed value along a sequence of good times. We then use the sequence of good times to obtain an efficient covering by approximate squares. A Misiurewicz-type argument (see \cite{Misiurewicz}) based on \cite{non-uniformly hyperbolic} is then used to extract a conditional $n$-th level Bernoulli measure for each of the horizontal fibers from the covering. Note that Misiurewicz's argument must be adjusted to deal with the lack of compactness. By weighting the horizontal fibres according to a Bernoulli measure derived from the frequencies of certain digits, along a subsequence of good times, we obtain an $n$-th level Bernoulli measure which not only has dimension close to the exponent given by the covering, but also integrates each of the potentials to approximately the correct value. In section \ref{Proof UB} we apply a series of approximation arguments to deduce the upper bound given in Theorem \ref{Main Countable} from the upper estimate from Section \ref{locally constant}.

The prove the lower bound we use the technique of concatenating measures applied by Gelfert and Rams in \cite{Gelfert Rams}. For each $m \in \N$ we obtain an compactly supported ergodic measure, with near optimal dimension, which integrates each of the first m potentials to approximately the required value. By carefully concatenating a sequence of such measures it is possible to obtain a measure for which typical points, with respect to that measure, have local dimension equal to the expression in Theorem \ref{Main Countable} and for which each of the countably many Birkhoff averages converge to the required value.

\section{The upper bound for locally constant potentials}\label{locally constant}

In this section we shall make the following simplifying assumptions. Firstly, we will suppose that there exists a contraction ratio $\zeta \in (0,1)$ such that for each $i \in \A$, $\sup_{x \in I}|g'_{i}(x)|\leq \zeta$. Secondly, we will suppose that we have finitely many  potentials, $\varphi^1, \cdots, \varphi^K$, each of which is both locally constant and bounded below by $1$. That is, for each $k=1, \cdots, K$, there exists a $\D$-sequence $(\varphi^k_{ij})_{(i,j) \in \D}$ such that $\varphi^k(\omega)= \varphi^k_{\omega_1}\geq 1$ for all $\omega=(\omega_{\nu})_{\nu \in \N} \in \Sigma$.

We shall often view the K-tuple of potentials, $\varphi^1, \cdots, \varphi^K$ as a single vector valued potential $\varphi: \omega \mapsto (\varphi_{\omega_1}^k)_{k=1}^K$, taking values in $\R^K$. We endow $\R^K$ with the supremum metric, which we shall denote by $|| \cdot ||_{\infty}$, as well as the usual partial order given by $(c_k)_{k=1}^K \leq (d_k)_{k=1}^K$ if and only if $c_k\leq d_k$ for all $k=1,\cdots, K$. We also let $[c,d]:= \left\lbrace x \in \R^K: c \leq x \leq d \right\rbrace$.

Let $\R\cup\{-\infty, +\infty\}$ denote the usual two-point compactification of $\R$. Given a sequence of real numbers $(a_n)_{n\in\N}$ we let $\Omega((a_n)_{n \in \N})$ denote its set of accumulation points in $\R\cup\{-\infty, +\infty\}$. For each $k=1,\cdots, K$, we fix some (possibly infinite) interval $\Gamma_k=[\gamma^k_{\min},\gamma^k_{\max}]  \subseteq \R\cup\{+\infty\}$, let $\Gamma:=\prod_{k=1}^K\Gamma_k=[\gamma_{\min}, \gamma_{\max}]$, where $\gamma_{\min}:=(\gamma_{\min}^k)_{k=1}^K$ and  $\gamma_{\max}:=(\gamma_{\max}^k)_{k=1}^K$. Define,
\begin{equation}
E_{\varphi}(\Gamma):=\left\lbrace \omega \in \Sigma :\Omega((A_n(\varphi)(\omega))_{n \in \N}) \subseteq \Gamma \right\rbrace,
\end{equation}
and let $J_{\varphi}(\Gamma):= \Pi(E_{\varphi}(\Gamma))$.

For each $(i,j) \in \D$ we let $b_i:= \sup_{x \in I}|g'_{i}(x)|$ and $a_{ij}:=\sup_{x \in I}|f'_{ij}(x)|$. We define $\tilde{\chi}:\Sigma \rightarrow \R$ and $\tilde{\chi}^v: \Sigma_v \rightarrow \R$ by
\begin{eqnarray}
\tilde{\chi}(\omega)&:=& -\log a_{\omega_1} \text{   for   } \omega=(\omega_{\nu})_{\nu \in \N}\in \Sigma,\\
\tilde{\psi}(\mathbf{i})&:=& -\log b_{i_1}\text{   for   } \mathbf{i}=(i_{\nu})_{\nu \in \N}\in \Sigma_v.
\end{eqnarray}
Given $q \in \N$ and $\mu \in \M^*_{\sigma^q}(\Sigma)$ we define
\begin{eqnarray}
\tilde{D}_q(\mu)&:=& \frac{h_{\mu}(\sigma^q|\pi^{-1}\mathscr{A})}{\int S_q(\tilde{\chi}) d \mu }+\frac{ h_{\mu\circ \pi^{-1}}( \sigma_v^q)}{\int S_q(\tilde{\psi}) d \mu\circ \pi_v^{-1}}.
\end{eqnarray}

\begin{prop} \label{locally constant prop}
\[\dim J_{\varphi}(\alpha)\leq \lim_{\xi \rightarrow 0}\left\lbrace \tilde{D}_q(\mu): q \in \N, \mu \in \E_{\sigma^q}^*(\Sigma), \int A_q(\varphi) d\mu \in [\gamma_{\min}-\xi, \gamma_{\max}+\xi]\right\rbrace.\]
\end{prop}
\subsection{Building a Cover}

Define
\begin{equation}
L_n(\omega):=\min \left\lbrace l \geq 1 : \prod^l_{\nu=1} b_{i_{\nu}} \leq \prod^n_{\nu=1} a_{i_{\nu}j_{\nu}}\right\rbrace .
\end{equation}
Note that this implies
\begin{equation} \label{Lyapunov Ratios}
1 \leq \frac{\prod_{\nu=1}^{n} a_{i_{\nu} j_{\nu}}}{\prod_{\nu=1}^{L_n(\omega)} b_{i_{\nu}}} < b_{\min}^{-1} .
\end{equation}
Moreover, since $a_{ij}\leq b_i$ for all $(i,j) \in \D$, $L_n(\omega) \geq n$.

Given $(\omega_{\nu})_{\nu=1}^n=((i_{\nu},j_{\nu}))_{\nu =1}^n\in \D^n$ we let
\begin{equation}
[\omega_{1} \cdots \omega_n]:=\{ \omega' \in \Sigma: \omega'_{\nu}= \omega_{\nu} \text{ for } \nu=1, \cdots, n\}
\end{equation}
and 
\begin{equation}
[i_{1} \cdots i_n]:=\{ \mathbf{i}' \in \Sigma_v: i'_{\nu}= i_{\nu} \text{ for } \nu=1, \cdots, n\}.
\end{equation}
Given $\omega= ((i_{\nu}, j_{\nu}))_{\nu =1}^{\infty}\in \Sigma$ we let $B_n(\omega)$ denote the $n$th approximate square, \begin{equation}
B_n(\omega):=\Pi([\omega_1 \cdots \omega_{n} ]\cap \sigma^{-n}\pi^{-1}[i_{n+1} \cdots i_{L_n(\omega)}]). 
\end{equation}
Thus,
\begin{equation}
\diam(B_n(\omega))\leq \max\left\lbrace \left(\prod_{\nu=1}^{n} a_{i_{\nu} j_{\nu}}\right), b_{\min}^{-1}\left(\prod_{\nu=1}^{L_n(\omega)} b_{i_{\nu}}\right)\right\rbrace.
\end{equation}
We also define a map $\phi_i: \bigcup_{n \in \N\cup\{0\}} \D^n \rightarrow \{i\} \times \bigcup_{n \in \N\cup\{0\}} \B_i^n$ for each $i \in \A$ by 
\[\phi_i:((i'_1,j'_{1}), (i'_1,j'_{2}), \cdots, (i'_1,j'_{n})) \mapsto ((i,j'_{\nu_1}), (i,j'_{\nu_2}), \cdots, (i,j'_{\nu_{n_i}})),\]
where $\nu_1 <\nu_2<\cdots <\nu_{n_i}$ and $\left\lbrace \nu_l \right\rbrace_{l=1}^{n_i}=\left\lbrace r \leq n: i'_r=i \right\rbrace$.

Given $q \in \N$ we define,
\begin{eqnarray*}
\Prob_q(\A):&=& \left\lbrace (p_{i})_{i \in \A^q}\in [0,1]^{\A^q}: \sum_{i \in \A^q}p_{i}=1\right\rbrace,\\
\Q_q(\A):&=& \left\lbrace (p_{i})_{i \in \A^q}\in \Prob_q(\A): p_{i}\in \Q\backslash\{0\} \text{ for each }i \in \A^q \right\rbrace.
\end{eqnarray*} 
Each $\Prob_q(\A)$ is given the maximum norm $||\cdot||_{\infty}$. Note that for each $q \in \N$, $\Prob_q(\A)$ is compact and $\Q_q(\A)$ is a dense countable subset. We let $\Prob(\A):=\Prob_1(\A)$ and $\Q(\A):=\Q_1(\A)$.

Given $\mathbf{p}=(p_{i})_{i \in \A^q}\in \Prob_q(\A)$  we define,
\begin{equation}
d_q(\mathbf{p}):=\frac{ \sum_{i\in \A^q} p_{i} \log p_i}{\sum_{i \in \A^q} p_{i}\log b_{i}}=\frac{ h_{\mu_{\p}}( \sigma_v^q)}{\int S_q(\tilde{\psi}) d \mu_{\p}},
\end{equation}
where $\mu_{\p}$ denotes the $q$-th level Bernoulli measure on $\Sigma_v$ defined by $\mu_{\p}([i])=p_i$ for all $i \in \A^q$. We let $d(\p):=d_1(\p)$ for $\p \in \Prob(\A)$.

Given $\rho \in \Q(\A)$, $n \in \N$, and $\lambda =(\lambda^k)_{k=1}^K \in \Q(\A)^K$ with $\lambda^k=(\lambda_i^k)_{i \in \A}$ for each $k=1, \cdots, K$ we define,
\begin{equation}
\B^{n,\epsilon}_{i}(\Gamma, \rho, \lambda):=\left\lbrace (ij_{\nu})_{\nu=1}^l: l =\rho_i n \pm \epsilon n, \sum_{\nu=1}^l\varphi^k_{ij_{\nu}} \pm \epsilon n \in \lambda^k_i \Gamma \right\rbrace,
\end{equation}
for each $i \in \A$ and let
\begin{equation}
\B^{n,\epsilon}(\Gamma, \rho, \lambda):=\left\lbrace (\vartheta^i)_{i\in \A} \in \prod_{i \in \A}\B^{n,\epsilon}_{i}(\Gamma, \rho, \lambda) : \sum_{i \in \A}|\vartheta^i|=n  \right\rbrace.
\end{equation}

Now define,
\begin{eqnarray*}
s_{n,\epsilon}(\Gamma,\rho, \lambda)&:=&\inf \left\lbrace s: \sum_{(\vartheta^i)_{i \in \A} \in \B^{n,\epsilon}(\Gamma, \rho, \lambda)} \prod_{i \in \A}a_{\vartheta^i}^s  \leq 1 \right\rbrace,\\
s_{\epsilon}(\Gamma,\rho, \lambda)&:=& \limsup_{n \rightarrow \infty} s_{n,\epsilon}(\Gamma,\rho, \lambda),\\
\delta_{\epsilon}(\Gamma)&:=&\sup \left\lbrace s_{\epsilon}(\Gamma,\rho, \lambda)+d(\rho): \rho \in \Q(\A), \lambda \in \Q(\A)^K \right\rbrace,\\
\delta(\Gamma)&:=& \liminf_{\epsilon \rightarrow 0} \delta_{\epsilon}(\Gamma).
\end{eqnarray*}

\begin{lemma}[Building a Cover]\label{Building a Cover} $\dim J_{\varphi}(\Gamma) \leq \delta(\Gamma).$
\end{lemma}

\begin{proof} Take some $\xi>0$. Note that the map $\p \mapsto d(\p)$ defines a continuous function on the compact space $\Prob(\A)$. Consequently there exists some $\epsilon>0$ such that $\delta_{\epsilon}(\Gamma)< \delta(\Gamma)+\xi$ and for all $\p, \q \in \Prob(\A)$ with $||\p-\q||_{\infty}<\epsilon$ we have $|d(\p)-d(\q)|<\xi$.

We shall define a function $F_{\xi}: \Sigma \rightarrow \Q(\A)^{2+K}$ in the following way. Given $\omega \in \Sigma$ we extract a subsequence $(n_q)_{q \in \N}$ satisfying,
\begin{enumerate}\vspace{3mm}
\item[(i)] $ \displaystyle\lim_{q\rightarrow \infty}\frac{\sum_{i \in \A}P_i(\omega|n_q)\log P_i(\omega|n_q)}{\sum_{i \in \A}P_i(\omega|n_q)\log b_i} =\limsup_{n\rightarrow \infty}\frac{\sum_{i \in \A}P_i(\omega|n)\log P_i(\omega|n)}{\sum_{i \in \A}P_i(\omega|n)\log b_i}$,
\vspace{3mm}
\item[(ii)] $ \displaystyle\lim_{q \rightarrow \infty}(P_i(\omega|n_q))_{i \in \A} = P(\omega)=(P_i(\omega))_{i \in \A}$,
\vspace{3mm}
\item[(iii)] $ \displaystyle\lim_{q \rightarrow \infty}(P_i(\omega|L_{n_q}(\omega)))_{i \in \A} = Q(\omega)= (Q_i(\omega))_{i \in \A}$,
\vspace{3mm}
\item[(iv)] $ \displaystyle\lim_{q\rightarrow \infty}\left(\frac{\sum_{j \in \B_i}P_{ij}(\omega|n_q) \varphi_{ij}^k}{\sum_{(i',j') \in \D}P_{i'j'}(\omega|n_q)\varphi_{i'j'}^k}\right)_{i \in \A}=R^k(\omega)=(R^k_i(\omega))_{i \in \A}$,
\end{enumerate}
for each $k=1, \cdots, K$. We let $R(\omega):=(R^k(\omega))_{k=1}^K$.
Note that by (i) we always have $d(Q(\omega))\leq d(P(\omega))$. Since $\Q(\A)$ is dense in $\Prob(\A)$ we may choose $\kappa(\omega)=(\kappa_i(\omega))_{i \in \A} \in \Q(\A)$ so that $\kappa_i(\omega)>P_i(\omega)\zeta^{\xi}$ for each $i \in \A$. We choose $\rho(\omega) \in \Q(\A)$ and  $ \lambda(\omega) \in \Q(\A)^K$  so that $||P(\omega)-\rho(\omega)||< \epsilon$ and $||R(\omega)-\lambda(\omega)||_{\infty}< \epsilon$. Let $F_{\xi}(\omega):=(\rho(\omega), \kappa(\omega), \lambda(\omega))$.

Define, 
\begin{eqnarray*}
E_{\varphi}^{(\rho, \kappa, \lambda)}(\Gamma)&:=&E_{\varphi}(\Gamma)\cap F_{\xi}^{-1}(\rho,\kappa, \lambda),\\
J_{\varphi}^{(\rho,\kappa, \lambda)}(\Gamma)&:=&\Pi(E_{\varphi}^{(\rho, \kappa, \lambda)}(\Gamma)).
\end{eqnarray*}
Since $\Q(\A)^{2+K}$ is countable, to show that 
\begin{equation*}
\dim J_{\varphi}(\Gamma)\leq \delta(\Gamma)+ 6\xi,
\end{equation*}
it suffices to fix $(\rho, \kappa, \lambda)\in \Q(\A)^{2+K}$ and show that 
\begin{equation*}
\dim J_{\varphi}^{(\rho,\kappa,\lambda)}(\Gamma)\leq \delta(\Gamma)+ 6\xi.
\end{equation*}

By the definition of $s_{\epsilon}(\Gamma,\rho, \lambda)$ we may take $N(\epsilon)\in \N$ so that for all $n\geq N(\epsilon)$ we have
\begin{equation}
\sum_{(\vartheta^i)_{i \in \A} \in \B^{n,\epsilon}(\Gamma, \rho, \lambda)} \prod_{i \in \A}a_{\vartheta^i}^{s_{\epsilon}(\Gamma,\rho, \lambda)+\xi}  < 1,
\end{equation}
and hence,
\begin{equation}
\sum_{(\vartheta^i)_{i \in \A} \in \B^{n,\epsilon}(\Gamma, \rho, \lambda)} \prod_{i \in \A}a_{\vartheta^i}^{s_{\epsilon}(\Gamma,\rho, \lambda)+2\xi}  < \zeta^{n\xi}.
\end{equation}

Given $\omega \in E_{\varphi}^{(\rho, \kappa, \lambda)}(\Gamma)$ we have,
\begin{eqnarray*}
\lim_{q \rightarrow \infty} \frac{\sum_{i \in \A}P_i(\omega|L_{n_q}(\omega))\log \rho_i }{\sum_{i \in \A}P_i(\omega|L_{n_q}(\omega))\log b_i }& \leq & 
\lim_{q \rightarrow \infty} \frac{\sum_{i \in \A}P_i(\omega|L_{n_q}(\omega))\log P_i(\omega|L_{n_q}(\omega))}{\sum_{i \in \A}P_i(\omega|L_{n_q}(\omega))\log b_i }+\xi\\
&\leq & d(Q(\omega)) +\xi \leq  d(P(\omega))+\xi\\
&<& d(\rho)+2\xi.
\end{eqnarray*}
Thus, for all sufficiently large $q$ we have,
\begin{eqnarray*}
\diam(B_{n_q}(\omega))^{d(\rho)+3 \xi} &\leq & b_{\min}^{-d(\rho)-2\xi}\left(b_{i_1}\cdots b_{i_{L_n(\omega)}}\right)^{d(\rho)+2\xi}\\ 
&\leq& b_{\min}^{-d(\rho)-2\xi}\left(\rho_{i_1}\cdots \rho_{i_{L_{n_q}(\omega)}}\right)\zeta^{L_{n_q}(\omega)\xi}.
\end{eqnarray*}
We also have,
\begin{eqnarray*}
\diam(B_{n_q}(\omega)) &\leq &  \prod_{\nu=1}^{n_q} a_{i_{\nu} j_{\nu}}
\leq \prod_{i \in \A}a_{\phi^i(\omega|n_q)}.
\end{eqnarray*}
Moreover, by (2) and (4) we also have $\phi^i(\omega|n_q) \in \B^{n_q,\epsilon}_{i}(\Gamma, \rho, \lambda)$ for each $i \in \A$ and hence $(\phi^i(\omega|n_q))_{i \in \A} \in \B^{n_q,\epsilon}(\Gamma, \rho, \lambda)$ for all sufficiently large $q$. 

Thus, if we fix some $r>0$, then for each $\omega \in E_{\varphi}^{(\rho, \kappa, \lambda)}(\Gamma)$ we may take some $n(\omega)\geq N(\epsilon)$ so that,
\begin{enumerate}
\vspace{2mm}
\item[(i)] $\Pi(\omega) \in B_{n(\omega)}(\omega)$,
\vspace{4mm}
\item[(ii)] $\diam(B_{n(\omega)}(\omega))\leq \gamma$,
\vspace{4mm}
\item[(iii)] $\diam(B_{n(\omega)}(\omega))^{d(\rho)+3 \xi} \leq b_{\min}^{-d(\rho)-2\xi}\left(\rho_{i_1}\cdots \rho_{i_{L_{n(\omega)}(\omega)}}\right)\zeta^{L_{n(\omega)}(\omega)\xi},$
\vspace{4mm}
\item[(iv)] $\diam(B_{n_q}(\omega))^{s_{\epsilon}(\Gamma,\rho, \lambda)+2\xi} \leq \prod_{i \in \A}a_{\phi^i(\omega|n_q)}^{s_{\epsilon}(\Gamma,\rho, \lambda)+2\xi},$
\vspace{4mm}
\item[(v)] $(\phi^i(\omega|n(\omega)))_{i \in \A} \in \B^{n(\omega),\epsilon}(\Gamma, \rho, \lambda)$. 
\vspace{2mm}
\end{enumerate}
Let $\mathcal{B}_{r}:=\left\lbrace B_{n(\omega)}(\omega): \omega \in E_{\varphi}^{(\rho, \kappa, \lambda)}(\Gamma)\right\rbrace$. By (i) and (ii) above, $\mathcal{B}_r$ forms a countable $r$-cover of $J_{\varphi}^{(\rho, \kappa, \lambda)}(\Gamma)$. 

Note also that given $\omega^1=((i^1_{\nu},j^1_{\nu}))_{\nu\in \N}, \omega^2=((i^2_{\nu},j^2_{\nu}))_{\nu\in \N} \in  E_{\varphi}^{(\rho, \kappa, \lambda)}(\Gamma)$ with $(i^1_1, \cdots, i^1_{L_{n(\omega^1)}(\omega)})=(i^1_2, \cdots, i^2_{L_{n(\omega^2)}(\omega)})$ and $(\phi^i(\omega^1|n_(\omega)))_{i \in \A}=(\phi^i(\omega^2|n_(\omega^2)))_{i \in \A}$ we must have $B_{n(\omega^1)}(\omega^1)=B_{n(\omega^2)}(\omega^2)$. Hence,
\begin{eqnarray*}
&& \sum_{B \in \mathcal{B}_{\gamma}}\diam(B)^{s_{\epsilon}(\Gamma,\rho, \lambda)+d(\rho)+5\xi},\\
&\leq& b_{\min}^{-d(\rho)-2\xi} \sum_{l \in \N} \zeta^{l\xi}\left(\sum_{(i_1,\cdots, i_l)\in \A^l} \rho_{i_1}\cdots \rho_{i_l}\right) \times \left( \sum_{n \geq N(\epsilon)} \sum_{(\vartheta^i)_{i \in \A} \in \B^{n,\epsilon}(\Gamma, \rho, \lambda)} \prod_{i \in \A}a_{\vartheta^i}^{s_{\epsilon}(\Gamma,\rho, \lambda)+2\xi}\right),\\
&\leq & b_{\min}^{-d(\rho)-2\xi} \sum_{l \in \N} \zeta^{l\xi} \times \sum_{n \geq N(\epsilon)} \zeta^{n\xi}< \infty. 
\end{eqnarray*}
Letting $\gamma \rightarrow 0$ we have that
\begin{eqnarray*}
\dim J_{\varphi}^{(\rho,\kappa,\lambda)}(\Gamma) &\leq& s_{\epsilon}(\Gamma,\rho, \lambda)+d(\rho)+5\xi,\\
& \leq & \delta_{\epsilon}(\Gamma)+ 5 \xi,\\
& \leq & \delta(\Gamma)+ 6 \xi,
\end{eqnarray*}
by our choice of $\epsilon$. Since $\Q(\A)^{2+K}$ is countable, it follows that 
\begin{equation*}
\dim J_{\varphi}(\Gamma)\leq \delta(\Gamma)+ 6\xi.
\end{equation*}
Letting $\xi\rightarrow 0$ proves the lemma.
\end{proof}

\subsection{Constructing a Measure}

Define $\A^{n,\epsilon}(\Gamma, \rho)\subseteq \A^{\lceil (1+2\epsilon)n\rceil}$ by,
\begin{eqnarray*}
\A^{n,\epsilon}(\Gamma, \rho):= \left\lbrace \tau \in  \A^{\lceil (1+2\epsilon)n\rceil}: N_i(\tau)\geq (1+\epsilon)\rho_i n \text{ for each } i \in \A  \right\rbrace.
\end{eqnarray*}

\begin{lemma} \label{V prob lemma} Given $\rho \in \Prob(\A)$ there exists $M(\rho, \epsilon) \in \N$ such that for all $n \geq M(\rho, \epsilon)$ we have,
$P^{n,\epsilon}(\Gamma, \rho):= \sum_{\tau \in \A^{n,\epsilon}(\Gamma, \rho)} \rho_{\tau} >1-\epsilon$.
\end{lemma}
\begin{proof} Apply Kolmogorov's strong law of large numbers and then Egorov's theorem.
\end{proof}

\begin{lemma}[Constructing a Measure]\label{construct a measure} 
\[\delta(\Gamma) \leq \lim_{\xi \rightarrow 0}\left\lbrace \tilde{D}_q(\mu): q \in \N, \mu \in \E_{\sigma^q}^*(\Sigma), \int A_q(\varphi) d\mu \in [\gamma_{\min}-\xi, \gamma_{\max}+\xi]\right\rbrace.\]
\end{lemma}

\begin{proof}
We begin by fixing some $j_*^i \in \B_i$ for each $i \in \A$. We then let $a_*:=\min\left\lbrace a_{ij_*^i}: i \in \A \right\rbrace$ and $\varphi_*:=\max\left\lbrace \varphi^k_{ij_*^i}: i \in \A, k \leq K \right\rbrace$. In what follows we shall let $o(\epsilon)$ denote any quantity which depends only upon the observable $\varphi$, the iterated function system, our choice of $(j_*^i)_{i \in \A}$ and $\epsilon$, which tends to zero as $\epsilon$ tends to zero. Of course, the precise value of $o(\epsilon)$ will vary from line to line.

Take $\xi>0$. Then there exists $\epsilon_0(\xi)$ such that for all $\epsilon \leq \epsilon_0(\xi)<1/2$ we have $\delta_{\epsilon}(\Gamma)>\delta(\Gamma)-\xi$. Take $\epsilon \leq \epsilon_0(\xi)$. Then there exists $\rho \in \Q(\A)$ $\lambda \in \Q(\A)^K$ with $s(\Gamma, \rho,\lambda)+d(\rho)>\delta(\Gamma)-\xi$.

Consequently, there exists infinitely many $n\in\N$ for which  
\begin{eqnarray}\label{pre s prob measure}
\sum_{(\vartheta^i)_{i \in \A} \in \B^{n,\epsilon}(\Gamma, \rho, \lambda)} \prod_{i \in \A}a_{\vartheta^i}^{\delta(\Gamma)-d(\rho)-\xi}>1.
\end{eqnarray}
In particular we may apply Lemma \ref{V prob lemma} and take some such $n \geq M(\rho, \epsilon)$, so that $P^{n,\epsilon}(\Gamma, \rho)>1-\epsilon$. By (\ref{pre s prob measure}) there exists a finite subset $\F^{n,\epsilon}(\Gamma, \rho, \lambda) \subseteq \B^{n,\epsilon}(\Gamma, \rho, \lambda)$ and $s>\delta(\Gamma)-d(\rho)-\xi$ for which
\begin{eqnarray} \label{s prob measure}
\sum_{(\vartheta^i)_{i \in \A} \in \F^{n,\epsilon}(\Gamma, \rho, \lambda)} \prod_{i \in \A}a_{\vartheta^i}^{s}=1.
\end{eqnarray}

Recall that we defined $\A^{n,\epsilon}(\Gamma, \rho)\subseteq \A^{\lceil (1+2\epsilon)n\rceil}$ by,
\begin{eqnarray*}
\A^{n,\epsilon}(\Gamma, \rho):= \left\lbrace \tau \in  \A^{\lceil (1+2\epsilon)n\rceil}: N_i(\tau)\geq (1+\epsilon)\rho_i n \text{ for each } i \in \A  \right\rbrace.
\end{eqnarray*}

We now define an injective map $\eta:\A^{n,\epsilon}(\Gamma, \rho) \times \F^{n,\epsilon}(\Gamma, \rho, \lambda)\rightarrow \D^{\lceil (1+2\epsilon)n\rceil}$ so that for all $(\tau, (\vartheta_i)_{i \in \A})\in \A^{n,\epsilon}(\Gamma, \rho) \times \F^{n,\epsilon}(\Gamma, \rho, \lambda)$ and $i \in \A$ we have $\pi(\eta(\tau, (\vartheta_i)_{i \in \A}))=\tau$ and $\vartheta_i$ is an intial segment of $\phi^i(\eta(\tau, (\vartheta_i)_{i \in \A}))$. To define $\eta$ we proceed as follows. Take $(\tau, (\vartheta_i)_{i \in \A})\in \A^{n,\epsilon}(\Gamma, \rho) \times \F^{n,\epsilon}(\Gamma, \rho, \lambda)$. For each $i \in \A$ we write $\vartheta_i=((i,j^i_1), \cdots, (i,j_{m_i}^i))$ and let $\tau=(i_1, \cdots, i_{\lceil (1+2\epsilon)n\rceil})$. Now, for each $\nu \in \{1, \cdots, {\lceil (1+2\epsilon)n\rceil}\}$ we choose $i \in \A$ so that $i=i_{\nu}$, and choose $r$ so that $\nu$ is the $r$-th occurance of the digit $i$ in $\tau$. If $r \leq m_i$ then let $\eta_{\nu}:=(i,j^i_{r})$ and if $\nu>r$ let $\eta_{\nu}=(i,j^i_*)$. Write $\eta(\tau, (\vartheta_i)_{i \in \A})=(\eta_{\nu})_{\nu =1}^n$.
Note that 
\begin{eqnarray*} \prod_{i \in \A}a_{\vartheta_i} & \geq & a_{\eta(\tau, (\vartheta_i)_{i \in \A})} \\ 
&\geq& \prod_{i \in \A}\left(a_{\vartheta_i}\times a_{i j_*^i}^{3n\epsilon} \right) \geq \left(\prod_{i \in \A}a_{\vartheta_i}\right)\times a_{*}^{3\#\A n\epsilon}.
\end{eqnarray*}
That is,
\begin{eqnarray}\label{Z ref Lyap} \log a_{\eta(\tau, (\vartheta_i)_{i \in \A})} =\log \left(\prod_{i \in \A}a_{\vartheta_i}\right) + n o(\epsilon).
\end{eqnarray}
Similarly, for each $i \in \A$ and $k \in \{1, \cdots, K\}$ we have
\begin{eqnarray*}
\rho_i(1-\epsilon) \gamma^k_{\min}n &\leq & \sum_{j \in \B_i}N_{ij}(\vartheta_i)\varphi^k_{ij},\\
&\leq &\sum_{j \in \B_i}N_{ij}(\eta(\tau, (\vartheta_i)_{i \in \A}))\varphi^k_{ij},\\
&\leq & \sum_{j \in \B_i}N_{ij}(\vartheta_i)\varphi_{ij}+ 3\epsilon n \varphi^k_{ij_*^i},\\
&\leq & \rho_i(1+\epsilon)\gamma_{\max}n + 3\epsilon n \varphi_{*}.
\end{eqnarray*}
Hence,
\begin{eqnarray}
\sum_{(i,j) \in \D}P_{ij}(\eta(\tau, (\vartheta_i)_{i \in \A}))\varphi^k_{ij} \in [\gamma^k_{\min}-o(\epsilon),\gamma^k_{\max}+o(\epsilon)],
\end{eqnarray}
for each $k=1,\cdots, K$, and so,
\begin{eqnarray}\label{Birkhoff on well behaved section}
\sum_{(i,j) \in \D}P_{ij}(\eta(\tau, (\vartheta_i)_{i \in \A}))\varphi_{ij} \in [\gamma_{\min}-o(\epsilon),\gamma_{\max}+o(\epsilon)].
\end{eqnarray}

We define a compactly supported $\lceil (1+2\epsilon)n\rceil$-level Bernoulli measure $\nu$ on $\Sigma$ in the following way. First let $\nu(\pi^{-1}[\tau]):=\rho_{\tau}$ for each $\tau \in \A^{\lceil (1+2\epsilon)n\rceil}$. Then, given $\tau \in \A^{\lceil (1+2\epsilon)n\rceil}$ and $\kappa \in \A^{\lceil (1+2\epsilon)n\rceil}$ with $\pi(\kappa)=\tau$ either $\tau \in \A^{n,\epsilon}(\Gamma, \rho) $ in which case we let
\begin{eqnarray*}
\frac{\nu([\kappa])}{\nu(\pi^{-1}[\tau])}:=\begin{cases}\prod_{i \in \A}a_{\vartheta_i}^s\text{ if }\kappa= \eta(\tau, (\vartheta_i)_{i \in \A})\text{ for some  }(\vartheta_i)_{i \in \A}\in \prod_{i \in \A}\F^{n,\epsilon}_{i}(\Gamma, \rho, \lambda) ,\\0\text{ if otherwise,}\end{cases}
\end{eqnarray*}
or $\tau=(i_{\nu})_{\nu=1}^{\lceil (1+2\epsilon)n\rceil} \in \A^{\lceil (1+2\epsilon)n\rceil}\backslash \A^{n,\epsilon}(\Gamma, \rho)$, in which case we let
\begin{eqnarray*}
\frac{\nu([\kappa])}{\nu(\pi^{-1}[\tau])}:=\begin{cases}1\text{ if }\kappa=((i_{\nu},j_{*}^{i_{\nu}}))_{\nu=1}^{\lceil (1+2\epsilon)n\rceil} ,\\0\text{ if otherwise.}\end{cases}
\end{eqnarray*}

Since $\nu \circ \pi$ is the $(\rho_i)_{i \in \A}$ Bernoulli measure on $\A^{\N}$ we have
\begin{eqnarray*}
\frac{ h_{\nu\circ \pi_v^{-1}}( \sigma_v^{\lceil (1+2\epsilon)n\rceil})}{\int S_{\lceil (1+2\epsilon)n\rceil}(\tilde{\psi}) d \nu \circ \pi^{-1}}&=&\frac{\sum_{\tau \in \A^{\lceil (1+2\epsilon)n\rceil}}\rho_{\tau} \log \rho_{\tau}}{\sum_{\tau \in \A^{\lceil (1+2\epsilon)n\rceil}}\rho_{\tau} \log b_{\tau}}\\&=&\frac{\sum_{i \in \A}\rho_{i} \log \rho_{i}}{\sum_{i\in \A}\rho_{i} \log b_{i}}=d(\rho).
\end{eqnarray*}
Now define
\begin{eqnarray*}
\Z^{n,\epsilon}(\Gamma, \rho, \lambda)&:=& \sum_{(\vartheta_i)_{i \in \A}\in \F^{n,\epsilon}(\Gamma, \rho, \lambda)} \prod_{i \in \A}a_{\vartheta_i}^s \log \left(\prod_{i \in \A}a_{\vartheta_i}\right),
\end{eqnarray*}
By (\ref{s prob measure}) together with the fact that $\log a_{ij} \leq \log \zeta<0$ for all $(i,j) \in \D$ we have,
\begin{eqnarray*}
\Z^{n,\epsilon}(\Gamma, \rho, \lambda)&\leq & n\log \zeta.
\end{eqnarray*}

By the definition of $\nu$ we have, 
\begin{eqnarray*}
h_{\nu}(\sigma^{\lceil (1+2\epsilon)n\rceil}|\pi^{-1}\mathscr{A})&=&\left(\sum_{\tau \in \A^{n,\epsilon}(\Gamma, \rho)} \rho_{\tau}\right)\sum_{(\vartheta_i)_{i \in \A}\in \F^{n,\epsilon}(\Gamma, \rho, \lambda)} \prod_{i \in \A}a_{\vartheta_i}^s \log \left(\prod_{i \in \A}a_{\vartheta_i}^s\right)\\
&=& P^{n,\epsilon}(\Gamma, \rho)\times s \Z^{n,\epsilon}_{i}(\Gamma, \rho, \lambda).
\end{eqnarray*}

Also, by (\ref{Z ref Lyap}) we have
\begin{eqnarray*}
\int S_{\lceil (1+2\epsilon)n\rceil}(\tilde{\chi}) d \nu  &=&  \left(\sum_{\tau \in \A^{n,\epsilon}(\Gamma, \rho)} \rho_{\tau}\right)\left(\sum_{(\vartheta_i)_{i \in \A}\in \F^{n,\epsilon}(\Gamma, \rho, \lambda)} \prod_{i \in \A}a_{\vartheta_i}^s \log \left(\prod_{i \in \A}a_{\vartheta_i}\right)-n o(\epsilon)\right)\\ & &+  \left(1-\sum_{\tau \in \A^{n,\epsilon}(\Gamma, \rho)} \rho_{\tau}\right) \lceil (1+2\epsilon)n\rceil \log a_{*} \\
&=& P^{n,\epsilon}(\Gamma, \rho)\left(\Z^{n,\epsilon}_{i}(\Gamma, \rho, \lambda)-n o(\epsilon)\right)+  \left(1-P^{n,\epsilon}(\Gamma, \rho)\right) \lceil (1+2\epsilon)n\rceil \log a_{*}.
\end{eqnarray*}
Since $n \geq M(\rho, \epsilon)$ we have $P^{n,\epsilon}(\Gamma, \rho)>1-\epsilon$  and consequently, 
\begin{eqnarray*}
\frac{h_{\nu}(\sigma^{\lceil (1+2\epsilon)n\rceil}|\pi^{-1}\mathscr{A})}{\int S_{\lceil (1+2\epsilon)n\rceil}(\tilde{\chi}) d \nu}\geq \frac{s}{1-o(\epsilon)}.
\end{eqnarray*}
Combining this with the fact that $s+d(\rho)>\delta(\Gamma)-\xi$ we have,
\begin{eqnarray}\label{Dimension estimate for nu}
\tilde{D}_{\lceil (1+2\epsilon)n\rceil}(\nu) &\geq & \frac{s}{1-o(\epsilon)}+d(\rho)\\
\nonumber &\geq & \frac{s+d(\rho)}{1-o(\epsilon)} \\
\nonumber & > & \frac{\delta(\Gamma)-\xi}{1-o(\epsilon)}.
\end{eqnarray}
Moreover, by the construction of $\nu$ combined with (\ref{Birkhoff on well behaved section}) we have,
\begin{eqnarray} \label{Lower Birkhoff bound nu}
\int A_{\lceil (1+2\epsilon)n\rceil}(\varphi) d\nu &\geq &  P^{n,\epsilon}(\Gamma, \rho)(\gamma_{\min}-o(\epsilon))\\
\nonumber &\geq &(1-\epsilon)(\gamma_{\min}-o(\epsilon)).
\end{eqnarray}
Similarly,
\begin{eqnarray}\label{Upper Birkhoff bound nu}
\int A_{\lceil (1+2\epsilon)n\rceil}(\varphi) d\nu &\leq &  P^{n,\epsilon}(\Gamma, \rho)(\gamma_{\max}+o(\epsilon))\\ \nonumber && \hspace{.5cm}+ \left(1- P^{n,\epsilon}(\Gamma, \rho)\right) \varphi_* \\
\nonumber &\leq & \gamma_{\max}+o(\epsilon).
\end{eqnarray}
Since we can obtain such a measure $\mu$ for all $\epsilon \leq \epsilon_0(\xi)$, the lemma follows by taking $\epsilon$ sufficiently small.
\end{proof}

\section{Approximation Arguments}\label{Proof UB}

In this section we apply Proposition \ref{locally constant prop} to obtain upper estimates of increasing generality until we obtain the upper bound in Theorem \ref{Main Countable}.

We begin by dropping the assumption that our potentials $\varphi$ are locally constant. Instead we assume that we have finitely many potentials $\varphi_1, \cdots, \varphi_K$, with finite first level variation $\var_1(\varphi_k)<\infty$, for each $k=1,\cdots, K$. We retain the assumption that for some $\zeta \in (0,1)$ we have $\sup_{x \in I}|g'_{i}(x)|\leq \zeta$ for each $i \in \A$ and also assume that $\var_1(\chi), \var_1(\psi)< - \log \zeta$. We define \begin{eqnarray*} C_{\sigma}(\chi, \psi):&=& \max \left\lbrace \frac{-\log \zeta}{-\log \zeta -\var_1(\chi)}, \frac{-\log \zeta}{-\log \zeta -\var_1(\psi)}\right\rbrace.
\end{eqnarray*} Proposition \ref{locally constant prop} gives the following estimate.

\begin{lemma}\label{approx lemma1} Suppose we have finitely many potentials $\varphi_1, \cdots, \varphi_M$, with $\var_1(\varphi_k)<\infty$. Suppose also that for some $\zeta \in (0,1)$ we have $\sup_{x \in I}|g'_{i}(x)|\leq \zeta$ for each $i \in \A$ and that $\var_1(\chi), \var_1(\psi)< - \log \zeta$. Suppose that $\alpha=(\alpha_k)_{k=1}^M$ is such that for all $k \leq K \leq M$ we have $\alpha_k \in \R$ and for $K < k \leq M$, $\alpha_k= \infty$. Then given any $m \in \N$,
\begin{eqnarray*}
\dim J_{\varphi}(\alpha) &\leq& C_{\sigma}(\chi, \psi) \sup \left\lbrace D_q(\mu)\right\rbrace,
\end{eqnarray*}
where the supremum is taken over all $\mu \in \E_{\sigma^q}^*(\Sigma)$ for some $q \in \N$ with $| \int A_q(\varphi_k) d\mu - \alpha_k |<  3\var_1(\varphi_k)$ for $k \leq K$ and $\int A_q(\varphi_k) d\mu>m$ for $K < k \leq M$.
\end{lemma}
\begin{proof}
For each $k = 1, \cdots, K$ we define a locally constant potential $\tilde{\varphi}^k$ by \begin{eqnarray}
\tilde{\varphi}_k(\omega):=\sup \left\lbrace \varphi_k(\omega'): \omega_1=\omega_1'\right\rbrace,
\end{eqnarray}
for all $\omega=(\omega_{\nu})_{\nu \in \N}\in \Sigma$. It follows that $|| \varphi_k -\tilde{\varphi_k}||_{\infty}<\var_1(\varphi_k)$. Thus, for all $\omega \in E_{\varphi}(\alpha)$ we have $\Omega(A_n(\tilde{\varphi}_k))\subseteq [\alpha_k - \var_1(\varphi_k), \alpha_k + \var_1(\varphi_k)]$ for $k \leq K$ and $\Omega(A_n(\tilde{\varphi}_k))= \{\infty\}$ for $K < k \leq M$, since $\lim_{n \rightarrow \infty}A_n(\varphi_k)(\omega)= \alpha_k$ for all $k \leq M$. Hence, $J_{\varphi}(\alpha)\subseteq J_{\tilde{\varphi}}(\Gamma) $ where $\Gamma:= \prod_{k=1}^K  [\alpha_k - \var_1(\varphi_k), \alpha_k + \var_1(\varphi_k)] \times \prod_{k=K+1}^M [m+2\var_1(\varphi_k), \infty]$. Thus, by Proposition \ref{locally constant prop} we have,
\begin{eqnarray*}
\dim J_{\varphi}(\alpha)&\leq& 
\dim J_{\tilde{\varphi}}(\Gamma) \\
& \leq & \lim_{\xi \rightarrow 0} \sup \bigg\{\tilde{D}_q(\mu): q \in \N, \mu \in \E_{\sigma^q}^*(\Sigma),  \big| \int A_q(\tilde{\varphi}_k) d\mu - \alpha_k \big|< \var_1(\varphi_k)+\xi,\\ && \text{ for }k \leq K
\text{ and } \int A_q(\tilde{\varphi}_k) d\mu >m+2\var_1(\varphi_k)-\xi   \text{ for } K < k \leq M\bigg\}\\
& \leq & \sup \bigg\{\tilde{D}_q(\mu): q \in \N, \mu \in \E_{\sigma^q}^*(\Sigma),  \big| \int A_q(\tilde{\varphi}_k) d\mu - \alpha_k \big|< 2\var_1(\varphi_k),\\ && \text{ for }k \leq K
\text{ and } \int A_q(\tilde{\varphi}_k) d\mu>m+\var_1(\varphi_k)   \text{ for } K < k \leq M\bigg\}\\
& \leq & \sup \bigg\{\tilde{D}_q(\mu): q \in \N, \mu \in \E_{\sigma^q}^*(\Sigma),  \big| \int A_q(\varphi_k) d\mu - \alpha_k \big|< 3\var_1(\varphi_k),\\ && \text{ for }k \leq K
\text{ and } \int A_q(\varphi_k) d\mu>m   \text{ for } K < k \leq M\bigg\}.
\end{eqnarray*}
It is clear from the definitions of $\tilde{\chi}: \Sigma \rightarrow \R$ and $\zeta$ that $\tilde{\chi}(\omega) \geq \chi(\omega) -\var_1(\chi)(\omega)$ and $\chi(\omega) \geq - \log \zeta$  for all $\omega \in \Sigma$. Thus $\int S_q(\tilde{\chi}) d \mu \geq \int S_q(\chi) d \mu - q \var_1(\chi)>0$ for each $\mu \in \M_{\sigma^q}^*(\Sigma)$, since $\var_1(\chi)< -\log \zeta$. It follows that for each $\mu \in \M_{\sigma^q}^*(\Sigma)$ we have 
\begin{eqnarray*}
\frac{\int S_q(\chi) d \mu}{\int S_q(\tilde{\chi}) d \mu}\leq \frac{\int S_q(\chi) d \mu}{\int S_q(\chi) d \mu-q \var_1(\chi)} \leq \frac{ -\log \zeta}{-\log \zeta - \var_1(\chi)} \leq C_{\sigma}(\chi, \psi).
\end{eqnarray*}
Similarly for each 
 $\mu \in \M_{\sigma^q}^*(\Sigma)$ we have 
\begin{eqnarray*}
\frac{\int S_q(\psi) d \mu\circ \pi^{-1}}{\int S_q(\tilde{\psi}) d \mu\circ \pi^{-1}}\leq \frac{ -\log \zeta}{-\log \zeta - \var_1(\psi)} \leq C_{\sigma}(\chi, \psi).
\end{eqnarray*}
Recall that for each $\mu \in \M_{\sigma^q}^*(\Sigma)$ we defined, 
\begin{eqnarray}
D_q(\mu)&:=& \frac{h_{\mu}(\sigma^q|\pi^{-1}\mathscr{A})}{\int S_q(\chi) d \mu }+\frac{ h_{\mu\circ \pi^{-1}}( \sigma_v^q)}{\int S_q(\psi) d \mu\circ \pi^{-1}}\\
\tilde{D}_q(\mu)&:=& \frac{h_{\mu}(\sigma^q|\pi^{-1}\mathscr{A})}{\int S_q(\tilde{\chi}) d \mu }+\frac{ h_{\mu\circ \pi^{-1}}( \sigma_v^q)}{\int S_q(\tilde{\psi}) d \mu\circ \pi^{-1}}.
\end{eqnarray}
Thus, for each $\mu \in \M_{\sigma^q}^*(\Sigma)$ we have $\tilde{D}_q(\mu) \leq C_{\sigma}(\chi, \psi) D_q(\mu)$. The lemma follows.
\end{proof}
We now use the observation that an iterated N-system is itself an N-system to obtain a more refined estimate which applies in a more general situation. First recall that by the Uniform Contraction Condition, for each N system, there exists a contraction ratio $\zeta \in (0,1)$ and $N \in \N$ such that for all $n\geq N$ and all $\omega \in \D^n$ and all $\mathbf{i} \in \A^n$ we have 
\[\max\left\lbrace \sup_{x \in I}|f'_{\omega}(x)|,\sup_{x \in I}|g'_{\mathbf{i}}(x)| \right\rbrace \leq \zeta^n.\]  

For each $n \geq N$ we let \begin{eqnarray*} C^{n}_{\sigma}(\chi, \psi):&=& \max \left\lbrace \frac{-\log \zeta}{-\log \zeta -\var_n(A_n(\chi))}, \frac{-\log \zeta}{-\log \zeta -\var_n(A_n(\psi))}\right\rbrace.
\end{eqnarray*} 

\begin{lemma} \label{ Iterated INC lemma} Suppose we have finitely many potentials $\varphi_1, \cdots, \varphi_M$, with $\var_1(\varphi_k)<\infty$. Suppose that $\alpha=(\alpha_k)_{k=1}^M$ is such that for all $k \leq K \leq M$ we have $\alpha_k \in \R$ and for $K < k \leq M$, $\alpha_k= \infty$. Fix some $m \in \N$. Then for all sufficiently large $n \in \N$ we have,
\begin{eqnarray*}
\dim J_{\varphi}(\alpha) &\leq& C_{\sigma}^n(\chi, \psi) \sup \left\lbrace D_q(\mu)\right\rbrace,
\end{eqnarray*}
where the supremum is taken over all $\mu \in \E_{\sigma^q}^*(\Sigma)$ for some $q \in \N$ with $| \int A_q(\varphi_k) d\mu - \alpha_k |<  3\var_n(A_n(\varphi_k))$ for $k \leq K$ and $\int A_q(\varphi_k) d\mu>m$ for $K < k \leq M$.
\end{lemma}

\begin{proof} First note that by the Uniform Contraction Condition, together with the Tempered Distortion Condition applied to $\chi$, $\psi$ and $\varphi_1, \cdots, \varphi_K$ we may choose $N \in \N$ so that for all $n \geq N$ we have
\begin{enumerate}
\vspace{2mm}
\item[(i)] $\max\left\lbrace \sup_{x \in I}|f'_{\omega}(x)|,\sup_{x \in I}|g'_{\mathbf{i}}(x)| \right\rbrace \leq \zeta^n$,
\vspace{4mm}
\item[(ii)] $ \max \left\lbrace \var_n(A_n(\chi)), \var_n(A_n(\psi)) \right\rbrace < -\log \zeta$,
\vspace{4mm}
\item[(iii)] $ \max \left\lbrace \var_n(A_n(\varphi_k)):k \in \{1,\cdots, K\} \right\rbrace < \infty.$
\vspace{2mm}
\end{enumerate}
For each $n\geq N$ we construct an associated iterated function system in the following way. Given $\xi=\xi_1 \cdots \xi_n \in \D^n$ we let 
\begin{equation*}
S_{\xi}:= S_{\xi_1} \circ \cdots \circ S_{\xi_n}.
\end{equation*}
It follows from the fact that $(S_{ij})_{(i,j) \in \D}$ is an INC-system that $(S_{\eta})_{\eta \in \D^n}$ is also an INC- system. Moreover, it follows from conditions (i), (ii) and (iii) above that the potentials $A_n(\varphi_1), \cdots A_n(\varphi_K)$ on $(\D^n)^{\N}=\Sigma$, together with the INC-system $(S_{\eta})_{\eta \in \D^n}$ satisfy the conditions of Lemma \ref{approx lemma1} with $\sigma^n$ in place of $\sigma$, $A_n(\varphi_k)$ in place of $\varphi_k$,  $S_n(\chi)$ in place of $\chi$, $S_n(\psi)$ in place of $\psi$, $\zeta^n$ in place of $\zeta$ and $\var_n$ in place of $\var_1$. We let,
\begin{eqnarray*}
E^{\sigma^n}_{A_n(\varphi)}(\alpha)&:=& \left\lbrace \omega \in \Sigma: \lim_{l \rightarrow \infty}A_{ln}(\varphi_k)(\omega)=\alpha_k \text{ for all } k\leq K  \right\rbrace\\
J^{\sigma^n}_{A_n(\varphi)}(\alpha)&:=&\Pi\left(E_{A_n(\varphi)}(\alpha)\right).
\end{eqnarray*}
Note also that,
 \begin{eqnarray*} C_{\sigma^n}(S_n(\chi), S_n(\psi)):&=& \max \left\lbrace \frac{-\log \zeta^n}{-\log \zeta^n -\var_n(S_n(\chi))}, \frac{-\log \zeta^n}{-\log \zeta^n -\var_n(S_n(\psi))}\right\rbrace\\
&=& \max \left\lbrace \frac{-\log \zeta}{-\log \zeta -\var_n(A_n(\chi))}, \frac{-\log \zeta}{-\log \zeta -\var_n(A_n(\psi))}\right\rbrace\\
&=& C^{n}_{\sigma}(\chi, \psi).
\end{eqnarray*} 
Thus, by Lemma \ref{approx lemma1} we have, $\dim J^{\sigma^n}_{A_n(\varphi)}(\alpha)$
\begin{eqnarray*} 
 \leq C_{\sigma^n}(S_n(\chi), S_n(\psi)) & \sup & \bigg\{D_{nq}(\mu): q \in \N, \mu \in \E_{\sigma^{nq}}^*(\Sigma),  \big| \int A_{nq}(\varphi_k) d\mu - \alpha_k \big|< 3\var_n(A_n(\varphi_k)),\\ && \text{ for }k \leq K
\text{ and } \int A_{nq}(\varphi_k) d\mu>m   \text{ for } K < k \leq M\bigg\}\\
\leq C_{\sigma}^n(\chi,\psi) & \sup & \bigg\{D_{q}(\mu): q \in \N, \mu \in \E_{\sigma^{q}}^*(\Sigma),  \big| \int A_{q}(\varphi_k) d\mu - \alpha_k \big|< 3\var_n(A_n(\varphi_k)),\\ && \text{ for }k \leq K
\text{ and } \int A_{nq}(\varphi_k) d\mu>m   \text{ for } K < k \leq M\bigg\}.
\end{eqnarray*}
Moreover, given $\omega \in E_{\varphi}(\alpha)$ we have, $\lim_{l \rightarrow \infty}A_l(\varphi_k)(\omega)=\alpha_k$ and hence $\lim_{l \rightarrow \infty}A_{ln}(\varphi_k)(\omega)=\alpha_k$. Thus, $J_{\varphi}(\alpha) \subseteq J^{\sigma^n}_{A_n(\varphi)}(\alpha)$. Hence, 
\begin{eqnarray*} 
\dim J_{\varphi}(\alpha) \leq C_{\sigma}^n(\chi,\psi) & \sup & \bigg\{D_{q}(\mu): q \in \N, \mu \in \E_{\sigma^{q}}^*(\Sigma),  \big| \int A_{q}(\varphi_k) d\mu - \alpha_k \big|< 3\var_n(A_n(\varphi_k)), \\&& \text{ for }k \leq K
\text{ and } \int A_{nq}(\varphi_k) d\mu>m   \text{ for } K < k \leq M\bigg\}.\end{eqnarray*}
\end{proof}

We now require a lemma relating $\sigma^q$-invariant measures to $\sigma$-invariant measures.
\begin{lemma}\label{averagemeasure}
Take $\nu \in \M^*_{\sigma^q}(\Sigma)$ and let $\mu=A_q(\nu)$. Then, 
\begin{enumerate}
\vspace{2mm}
\item[(i)] $\mu\in\M^*_{\sigma}(\Sigma)$,
\vspace{4mm}
\item[(ii)] If $\nu\in\E^*_{\sigma^q}(\Sigma)$ then $\mu\in\E^*_{\sigma}(\Sigma)$,
\vspace{4mm}
\item[(iii)] $h_{\mu}(\sigma)=q^{-1} h_{\nu}(\sigma^k)$,
\vspace{4mm}
\item[(iv)] $h_{\mu\circ \pi^{-1}}(\sigma_v)=q^{-1}h_{\nu\circ \pi^{-1}}(\sigma_v^q)$,
\vspace{4mm}
\item[(v)] $h_{\mu}(\sigma|\pi_v^{-1}\mathscr{A}_v)=q^{-1}h_{\nu}(\sigma^q|\pi^{-1}\mathscr{A}).$
\vspace{2mm}
\end{enumerate}
Moreover, given any $\theta \in C(\Sigma)$, $\theta^v \in C(\Sigma_v)$ we have,
\begin{enumerate}
\vspace{2mm}
\item[(vi)] $\int \theta d \mu = \int A_q(\theta) d \nu$,
\vspace{4mm}
\item[(vii)] $\int \theta^v d \mu \circ \pi^{-1} = \int A_q(\theta^v) d \nu\circ \pi^{-1}$.
\end{enumerate}

\end{lemma}
\begin{proof}
Parts (i), (ii), (iii) and (vi) follow from \cite[Lemma 2]{non-uniformly hyperbolic}. It is clear that $\mu$ is compactly supported. Since $\pi \circ \sigma = \sigma_v \circ \pi$ we have $A_k(\nu\circ \pi^{-1})=A_k(\nu)\circ \pi^{-1}$ and hence (iv) and (vii) also follow from \cite[Lemma 2]{non-uniformly hyperbolic}. Part (v) follows from parts (iii) and (iv) combined with the Abramov Rokhlin formula \cite{Abramov Rokhlin}.
\end{proof}
The following proposition completes the proof of the upper bound.
\begin{prop} Suppose we have countably many potentials $(\varphi_k)_{k \in \N}$. Then, for all $\alpha=(\alpha_k)_{k \in \N}\in \left(\R\cup \{\infty\}\right) ^{\N}$ we have, 
\[\dim J_{\varphi}(\alpha) \leq \lim_{m \rightarrow \infty}\sup\left\lbrace D(\mu):  \mu \in \E_{\sigma}^*(\Sigma), \int \varphi_k d\mu \in B_{m}(\alpha_k)   \text{ for } k \leq m\right\rbrace.\]
\end{prop}
\begin{proof}
It suffices to show that for each $m \in \N$ we have,
\begin{eqnarray*}\label{suffices}
\dim J_{\varphi}(\alpha) \leq \sup\left\lbrace D(\mu):  \mu \in \E_{\sigma}^*(\Sigma), \int \varphi_k d\mu \in B_{m}(\alpha_k)   \text{ for } k \leq m\right\rbrace.
\end{eqnarray*}
Fix $m \in \N$. Without loss of generality we may assume that there are only $m$ potentials $\varphi_1, \cdots, \varphi_m$. If not, we consider the set,
\begin{eqnarray}
E^m_{\varphi}(\alpha)&:=& \left\lbrace \omega \in \Sigma: \lim_{n \rightarrow \infty} A_n(\varphi_k)= \alpha_k \text{ for } k \leq m \right\rbrace,\\
J^m_{\varphi}(\alpha)&:=& \Pi \left(E^m_{\varphi}(\alpha) \right), 
\end{eqnarray}
and note that $E_{\varphi}(\alpha) \subseteq E^m_{\varphi}(\alpha)$ and hence $\dim J_{\varphi}(\alpha) \leq \dim J^m_{\varphi}(\alpha)$.
Finally we may reorder our potentials so that there is some $K\leq m$ such that 
for all $k \leq K \leq M$ we have $\alpha_k \in \R$ and for $K < k \leq M$, $\alpha_k= \infty$. Now we are in precisely the position of Lemma \ref{ Iterated INC lemma}, so
\begin{eqnarray*} 
\dim J_{\varphi}(\alpha) \leq C_{\sigma}^n(\chi,\psi) & \sup & \bigg\{D_{q}(\mu): q \in \N, \mu \in \M_{\sigma^{q}}^*(\Sigma),  \big| \int A_{q}(\varphi_k) d\mu - \alpha_k \big|< 3\var_n(A_n(\varphi_k)), \\&& \text{ for }k \leq K
\text{ and } \int A_{nq}(\varphi_k) d\mu>m   \text{ for } K < k \leq M\bigg\}.\end{eqnarray*}
Since $\lim_{n \rightarrow \infty}\var_n( A_n(\chi))=\lim_{n \rightarrow \infty} \var_n(A_n(\psi))=\lim_{n \rightarrow \infty} \var_n(A_n(\varphi_k))=0$ for all $k \leq m$, and hence $\lim_{n \rightarrow \infty}C_{\sigma}^n(\chi, \psi)= 1$, we have,
\begin{eqnarray*} 
\dim J_{\varphi}(\alpha) \leq &\sup& \bigg\{D_{q}(\mu): q \in \N, \mu \in \E_{\sigma^{q}}^*(\Sigma),  \big| \int A_{q}(\varphi_k) d\mu - \alpha_k \big|< \frac{1}{m}, \\&& \text{ for }k \leq K
\text{ and } \int A_{q}(\varphi_k) d\mu>m   \text{ for } K < k \leq M\bigg\}.\end{eqnarray*}
Recall that for $\gamma \in \R\cup \{+\infty\}$ and $l \in \N$ we let
\begin{eqnarray}
B_l(\gamma):= \begin{cases} \left\lbrace x: |x - \gamma|< \frac{1}{l} \right\rbrace \text{ if } \gamma \in \R\\ (l , +\infty)\text{ if } \gamma =\infty.\end{cases}
\end{eqnarray}
So we may rewrite the above inequality as 
\begin{eqnarray*}\label{suffices}
\dim J_{\varphi}(\alpha) \leq \sup\left\lbrace D_q(\mu): q \in \N,  \mu \in \E_{\sigma^q}^*(\Sigma), \int A_q(\varphi_k) d\mu \in B_{m}(\alpha_k)   \text{ for } k \leq m\right\rbrace.
\end{eqnarray*}
Finally, by Lemma \ref{averagemeasure}, given $\nu \in \E_{\sigma^q}^*(\Sigma)$ we may choose $\mu \in \E_{\sigma}^*(\Sigma)$ with $D(\mu) = D_q(\nu)$ and $\int \varphi_k d\mu = \int A_q(\varphi_k) d\mu$. Thus,
\begin{eqnarray*}\label{suffices}
\dim J_{\varphi}(\alpha) \leq \sup\left\lbrace D(\mu):  \mu \in \E_{\sigma}^*(\Sigma), \int \varphi_k d\mu \in B_{m}(\alpha_k)   \text{ for } k \leq m\right\rbrace.
\end{eqnarray*}
This completes the proof of the upper bound.
\end{proof}

\section{Preliminary lemmas for the lower bound}\label{Preliminaries for lower bound}

\subsection{Dimension Lemmas}
In this section we shall relate the symbolic local dimension of a measure to the local dimension of its projection. This will enable us to apply the following standard lemma.

\begin{lemma}\label{dim lem} Let $\nu$ be a finite Borel measure on some metric space $X$. Suppose we have $J\subseteq X$ with $\nu(J)>0$ such that for all $x \in J$
 \[\liminf_{r \rightarrow 0}\frac{\log \nu( B(x,r))}{\log r} \geq d.\]Then $\dim J \geq d$.
\end{lemma}
\begin{proof}
See \cite{Falconer Techniques} Proposition 2.2.
\end{proof}

Given subsets $A, B \subseteq \R$ by $A \leq B$ we mean $x\leq y$ for all $x \in A$ and $y \in B$.
We shall say that the digit set $\D$ is wide if there exists $(i_1,j_1),(i_2,j_2)\in \D$ with $i_1=i_2$ and $j_1\neq j_2$. It follows that there exists an $i'\in \A$ together with pairs $ j^1_-, j^2_-,j^1_0,j^2_0, j^1_+, j^2_+ \in \N$ so that
\begin{eqnarray}\label{enclosed interval}
f_{i'j^1_-}\circ f_{i'j^2_-}([0,1])\leq f_{i'j^1_0}\circ f_{i'j^2_0}([0,1]) \leq f_{i'j^1_+}\circ f_{i' j^2_+}([0,1]).
\end{eqnarray}

Now define,
\begin{equation}
W_n(\omega):= \min\left\lbrace l>n+1: \omega_{l}=(i',j^1_0), \omega_{l+1}=(i',j^2_0) \right\rbrace-n,
\end{equation}
and
\begin{equation}
R_n(\omega):= \max\left\lbrace -\log \inf_{x \in [0,1]}|f_{\omega_{\nu}}'(x)|: \nu \leq n+W_n(\omega)+1 \right\rbrace.
\end{equation}

\begin{lemma} \label{SymbDimPDimH} Let $\mu$ be a finite Borel measure on $\Sigma$ supported on $\pi^{-1}(\{\tau\})$ for some $\tau \in \Sigma_v$. Let $\nu:= \mu \circ \Pi^{-1}$ the corresponding projection on $\Lambda$. Suppose $\D$ is wide. Then for all $x=\Pi(\omega) \in \Lambda$ with $\pi(\omega)=\tau$ and $\lim_{n\rightarrow \infty} R_n(\omega)n^{-1}=\lim_{n \rightarrow \infty}R_n(\omega) W_n(\omega)n^{-1}=0$, 
\vspace{2mm} 
\begin{equation*}
\liminf_{r \rightarrow 0}\frac{\log \nu( B(x,r))}{\log r}\geq \liminf_{n \rightarrow \infty}\frac{-\log \mu([\omega|n])}{S_n(\chi)(\omega)}.
\end{equation*}
\end{lemma}
\begin{proof} 
Suppose that $\D$ is wide and fix  $x=(x_1,x_2)=\Pi(\omega) \in \Lambda$ for some $\omega=((i_{\nu},j_{\nu}))_{\nu \in \N}\in  \pi^{-1}(\{\tau\})$ with
$\lim_{n\rightarrow \infty} R_n(\omega)n^{-1}=\lim_{n \rightarrow \infty}R_n(\omega) W_n(\omega)n^{-1}=0$. First note that since $\mu$ is supported on $\pi^{-1}(\{\tau\})$, $\nu$ is supported on $\R \times \{x_2\}$.
Define,
\begin{equation*}
 a_0:= \inf \left\lbrace  |f'_{d}(z)|:z \in [0,1], d \in \{(i',j^1_+),(i',j^2_+),(i',j^1_-),(i',j^2_-)\}\right\rbrace.
 \end{equation*}
Take $n \in \N$. By the definition of $W_n(\omega)$ the finite string $\omega_{n+W_n(\omega)}=(i',j^1_0)$ and $\omega_{n+W_n(\omega)+1}=(i',j^2_0)$. Now let,
\begin{eqnarray*}
\eta_0&:=& (\omega_{n+1}, \cdots ,\omega_{n+W_n(\omega)-1},(i',j^1_0),(i',j^2_0))\\
\eta_+&:=& (\omega_{n+1}, \cdots ,\omega_{n+W_n(\omega)-1},(i',j^1_+),(i',j^2_+))\\ 
\eta_-&:=& (\omega_{n+1}, \cdots ,\omega_{n+W_n(\omega)-1},(i',j^1_-),(i',j^2_-)).
\end{eqnarray*}
It follows from (\ref{enclosed interval}) that one of the following holds;
\begin{eqnarray*} f_{\omega|n}\circ
f_{\eta_-}([0,1]) & \leq & f_{\omega|n}\circ f_{\eta_0}([0,1])\leq f_{\omega|n}\circ
f_{\eta_+}([0,1]),\\
f_{\omega|n}\circ f_{\eta_+}([0,1])&\leq & f_{\omega|n}\circ f_{\eta_0}([0,1])\leq f_{\omega|n}\circ f_{\eta_-}([0,1]).
\end{eqnarray*}
Clearly each interval is contained within the interval $f_{\omega|n}([0,1])$. Moreover, it follows from the definitions of $W_n(\omega)$ and $a_0$ that both $f_{\eta_+}([0,1])$ and $f_{\omega|n}\circ
f_{\eta_-}([0,1])$ are of diameter at least $\inf_{x \in [0,1]}|f_{\omega|n}'(x)| e^{-W_n(\omega)R_n(\omega)}a_0^2$. Thus, since $x_1 \in f_{\omega|n}\circ f_{\eta_0}([0,1])=f_{\omega|n+W_n(\omega)+1}([0,1])$ and $x_2 \in g_{\mathbf{i}|n}([0,1])$ we have,
\begin{equation}
B\left(x_1, \inf_{z \in [0,1]}|f_{\omega|n}'(z)| e^{-W_n(\omega)R_n(\omega)}a_0^3 \right) \subseteq f_{\omega|n}((0,1)),
\end{equation}
and hence,
\begin{equation}
\nu\left(B\left(x, \inf_{z \in [0,1]}|f_{\omega|n}'(z)| e^{W_n(\omega)R_n(\omega)}a_0^3 \right) \right) \leq \mu \left([\omega|n]\right),
\end{equation}
since $\nu$ is supported on $\R \times \{x_2\}$.
So let 
\begin{equation}
r_n:= \inf_{z \in [0,1]}|f_{\omega|n}'(z)| e^{W_n(\omega)R_n(\omega)}a_0^3 .
\end{equation}
Choose $(z_n)_{n \in \N}\subset [0,1]$ so that $|f_{\omega|n}'(z_n)|=\inf_{z \in [0,1]}|f_{\omega|n}'(z)|$.
Note that for all $(\tau_1, \cdots, \tau_n) \in \D^n$, $\diam(f_{\tau_1} \circ \cdots f_{\tau_n}([0,1])) \leq \zeta^n$, so since the family $\left\lbrace \log |f_{ij}'|:(i,j) \in \D\right\rbrace$ is uniformly equicontinuous we have, 
\begin{eqnarray*}
\frac{1}{n}\log \inf_{z \in [0,1]}|f_{\omega|n}'(z)|&=&\frac{1}{n}\log|f_{\omega|n}'(z_n)|,\\
&=& \frac{1}{n}\sum_{l=1}^{n} \log |f'_{\omega_l}(f_{\omega_{l+1}}\circ \cdots \circ f_{\omega_n}(z_n))|,\\
&=& \frac{1}{n}\sum_{l=1}^{n} \left(\log |f'_{\omega_l}(\Pi(\sigma^l\omega))|+o(n-l)\right),\\
&=& \frac{1}{n}\sum_{l=1}^{n} \log |f'_{\omega_l}(\Pi(\sigma^l\omega))|+o(n)\\
&=& -\frac{1}{n}S_n(\chi)(\omega)+o(n).
\end{eqnarray*}
It follows that,
\begin{eqnarray*}
\lim_{n\rightarrow \infty}\frac{\log r_n}{-S_n(\chi)(\omega)}&=& \lim_{n\rightarrow \infty}\frac{\log \inf_{z \in [0,1]}|f_{\omega|n}'(z)| e^{W_n(\omega)R_n(\omega)}a_0^3}{S_n(\chi)(\omega)}\\
&=& \lim_{n\rightarrow \infty}\frac{-S_n(\chi)(\omega)+o(n)+W_n(\omega)R_n(\omega)}{-S_n(\chi)(\omega)}=1.
\end{eqnarray*}
Therefore,
\begin{equation*}
\liminf_{n \rightarrow \infty}\frac{\log \nu( B(x,r_n))}{\log r_n} \geq \liminf_{n \rightarrow \infty}\frac{\log \mu([\omega|n])}{-S_n(\chi)(\omega)}.
\end{equation*}
To conclude the proof of the lemma we observe that 
\begin{eqnarray*}
\lim_{n \rightarrow \infty}\frac{\log r_{n+1}}{\log r_n}&=&\lim_{n \rightarrow \infty}\frac{S_{n+1}(\chi)(\omega)}{S_{n}(\chi)(\omega)},\\
&=&\lim_{n \rightarrow \infty}\frac{S_{n}(\chi)(\omega)+\log |f'_{\omega_{n+1}}(\Pi(\sigma^{n+1}\omega))|}{S_{n}(\chi)(\omega)},\\
&=&\lim_{n \rightarrow \infty}\frac{S_{n}(\chi)(\omega)+O(R_n(\omega))}{S_{n}(\chi)(\omega)}=1.
\end{eqnarray*}
\end{proof}

We say that the digit set $\D$ is tall if there exists $(i_3,j_3),(i_4,j_4)\in \D$ with $i_3\neq i_4$. It follows that there exists pairs $i^1_-,i^2_-,i^1_0,i^2_0,i^1_+,i^2_+ \in \A$ so that, for all $x, y, z \in [0,1]$ we have,
\begin{eqnarray}
g_{i^1_-}\circ g_{i^2_-}(x)\leq g_{i^1_0}\circ g_{i^2_0}(y) \leq g_{i^1_+}\circ g_{i^2_+}(z).
\end{eqnarray}

Define
\begin{equation}
T_n(\tau):= \min\left\lbrace l>n: i_{l-1}=i^1_0, i_{l}=i^2_0 \right\rbrace-n.
\end{equation}

\begin{lemma} \label{SymbDimPDimV} Let $\mu$ be a finite Borel measure on $\Sigma_v$ and let $\nu:= \mu \circ \Pi_v^{-1}$ denote the corresponding projection. Suppose $\D$ is tall. Then for all $y=\Pi_v(\tau) \in \Pi_v(\Sigma_v)$ with $\lim_{n\rightarrow \infty} T_n(\tau)n^{-1}=0$, 
\vspace{2mm} 
\begin{equation*}
\liminf_{r \rightarrow 0}\frac{\log \nu( B(y,r))}{\log r}\geq \liminf_{n \rightarrow \infty}\frac{-\log \mu([\tau|n])}{S_n(\psi)(\tau)}.
\end{equation*}
\end{lemma}
\begin{proof}
Proceed as in Lemma \ref{SymbDimPDimH} with $b_{\min}:= \max\left\lbrace ||-\log g_{i}'||_{\infty}: i \in \A \right\rbrace$ in place of $W_n(\omega)$.
\end{proof}

\section{Convergence Lemmas}

\begin{lemma}\label{exists approximate measure}
Given $\mu \in \M_{\sigma}^*(\Sigma)$, $\epsilon>0$ and $m \in \N$ we may obtain $q \geq m$ and $\nu \in \B_{\sigma^q}^{\dagger}(\Sigma)$ satisfying,
\begin{enumerate}
 \vspace{2mm}
 \item[(i)] $\displaystyle \bigg|\frac{ h_{\nu \circ \pi^{-1}}( \sigma_v^q)}{\int S_q(\psi) d \nu\circ \pi^{-1}}- \frac{ h_{\mu \circ \pi^{-1}}( \sigma_v)}{\int\psi d \mu\circ \pi^{-1}} \bigg|<\epsilon$,
 \vspace{4mm}
 \item[(ii)] $\displaystyle \bigg|\frac{ h_{\nu}(\sigma^q|\pi^{-1}\mathscr{A})}{\int S_q(\chi) d \nu}-\frac{ h_{\mu}(\sigma|\pi^{-1}\mathscr{A})}{\int \chi d \mu}\bigg|<\epsilon$,
 \vspace{4mm}
 \item[(iii)] $\displaystyle \bigg|\int A_{q}(\varphi_k) d\nu-\int \varphi_k d\mu \bigg| < \epsilon$ for all $k \leq m$,
 \vspace{4mm}
\item[(iv)] $\displaystyle \var_{n}\left( A_{n}(\varphi_k)\right)< \frac{1}{m}$, for all $n\geq q$ and all $k \leq m$,
\vspace{4mm}
 \item[(v)] $\displaystyle \max \left\lbrace \var_{n}\left( A_{n}(\chi)\right), \var_{n}\left( A_{n}(\psi)\right)\right\rbrace < \frac{1}{m}$, for all $n\geq q(m)$.
 \vspace{2mm}
  \end{enumerate}
\end{lemma}
\begin{proof} First not that since $\lim_{q \rightarrow \infty} \var_{q}\left( A_{q}(\chi)\right)=\lim_{q \rightarrow \infty} \var_{q}\left( A_{q}(\psi)\right)=0$ and $\lim_{q \rightarrow \infty} \var_{q}\left( A_{q}(\varphi_k)\right)=0$ for all $k$ we may choose $q_0 \geq m$ so that for $n \geq q_0$, $\max \left\lbrace \var_{n}\left( A_{n}(\chi)\right), \var_{n}\left( A_{n}(\psi)\right)\right\rbrace < \frac{1}{m}$ and $\var_{n}\left( A_{n}(\varphi_k)\right)< \frac{1}{m}$ for $k \leq m$.

Fix $\mu \in \M_{\sigma}^*(\Sigma)$. Given $q\in\N$ we let $\tilde{\nu}_q \in B_{\sigma^q}^*(\Sigma)$ denote the $k$-th level approximation of $\nu$. That
is, given a cylinder $[\omega_1\cdots \omega_{nq}]$ of length $nq$ we let
\begin{equation}
\tilde{\nu}_q([\omega_1\cdots\omega_{nq}]):=\prod_{l=0}^{n-1}\nu([\omega_{lq+1}\cdots\omega_{lq+q}]).
\end{equation}
By the Kolmogorov-Sinai theorem (see \cite{Walters} Theorem 4.18) we then have,
\begin{eqnarray} \label{k Bernoulli convergence i} 
\lim_{q\rightarrow\infty}q^{-1} h_{\tilde{\nu}_q}(\sigma^q)&=&h_{\nu}(\sigma),\\ \label{k Bernoulli convergence ii}
\lim_{q\rightarrow\infty}q^{-1}  h_{\tilde{\nu}_q\circ \pi^{-1}}(\sigma_v^q)&=& h_{\nu\circ \pi^{-1}}(\sigma_v).
\end{eqnarray}
Combining these two limits and applying the Abramov Rohklin formula \cite{Abramov Rokhlin} gives,
\begin{eqnarray}
 \label{k Bernoulli convergence iii} \lim_{q\rightarrow\infty}q^{-1} h_{\tilde{\nu}_q}(\sigma^q|\pi^{-1}\mathscr{A})&=&h_{\nu}(\sigma).
\end{eqnarray}
Since $\mu$ and $\tilde{\nu}_q$ agree on cylinders of length $q$ we have $\big|\int A_q (\varphi_k) d\tilde{\nu}_q-\int A_q (\varphi_k) d\mu\big|\leq \var_q A_k(\varphi_q)$, which tends to zero with $q$ by the tempered distortion property. Moreover, $\int A_q (\varphi_k) d\mu=\int \varphi_k d\mu$, since $\mu$ is $\sigma$-invariant. Hence,
\begin{eqnarray} \label{k Bernoulli convergence vi}
\lim_{k\rightarrow\infty}A_q(\varphi_k) d \tilde{\nu}_q  &=&\int \varphi d\nu,
\end{eqnarray}
for all $k \leq m$. The same argument also gives,
\begin{eqnarray} \label{k Bernoulli convergence iv}
\lim_{k\rightarrow\infty}A_q(\chi) d \tilde{\nu}_q  &=&\int \chi d\nu,\\ \label{k Bernoulli convergence v}
\lim_{k\rightarrow\infty}A_q(\psi) d \tilde{\nu}_q\circ \pi^{-1}  &=&\int \psi d\nu \circ \pi^{-1}.
\end{eqnarray}
Consequently, by taking $q \geq q_0$ sufficiently large, we may obtain $q \in \N$ and $\tilde{\nu} \in \B_{\sigma^q}^*(\Sigma)$ satisfying (i), (ii) and (iii) from the lemma. To obtain $\nu \in \B_{\sigma^q}^{0}(\Sigma)$ satisfying (i), (ii) and (iii) we  peturb $\tilde{\nu}$ slightly to obtain $\nu \in \B_{\sigma^q}^*(\Sigma)$  with   $\nu_q([\omega_1\cdots\omega_q])>0$ for
each $(\omega_1,\cdots,\omega_q)\in \D_0^q$ whilst using continuity to insure that (i), (ii) and (iii) still hold. Since $q \geq q_0$, (iv) and (v) also hold.
\end{proof}

Recall that we defined $\mathscr{A}$ to be the Borel sigma algebra on $\Sigma_v$. Given any Borel probability measure $\nu \in \M(\Sigma)$ and $\omega \in \Sigma$ we let $\nu^{\pi^{-1}\mathscr{A}}_{\omega}$ denote the conditional measure at $\omega$ \cite[Section 5.3]{ErgodicNumberTheory}. Since $\pi^{-1}\mathscr{A}$ is countably generated there exists $\Sigma' \subseteq \Sigma$ with $\nu(\Sigma')=1$ such that for all $\omega^1, \omega^2 \in \Sigma'$ with $\tau=\pi(\omega^1)=\pi(\omega^2)$ we have $\nu^{\pi^{-1}\mathscr{A}}_{\omega^1}=\nu^{\pi^{-1}\mathscr{A}}_{\omega^2}$ and $\nu^{\pi^{-1}\mathscr{A}}_{\omega^1}\left( \pi^{-1}\{\tau\}\right)=1$ \cite[Theorem 5.14]{ErgodicNumberTheory}. It follows that we can take a family of measures $\{\nu_{\tau}\}_{\tau \in \Sigma_v}\subset \M(\Sigma)$ with $\nu^{\tau} \left( \pi^{-1}\{\tau\}\right)=1$ for all $\tau \in \Sigma_v$ and
\begin{equation}
\nu=\int \nu^{\tau} d\nu \circ \pi^{-1}(\tau).
\end{equation}

We shall make use of the following well known result which is essentially contained within \cite[Lemma 9.3.1]{LYMED2}. We refer the reader to \cite{Parry}.

\begin{lemma}[Ergodic theorem of information theory]\label{Ergodic theorem of information theory} Suppose $T$ is an ergodic measure-preserving transformation of a Borel probability space $(X,\mathscr{B}, m)$. Let $\xi$ be a countable partition with $\bigvee^{\infty}_{i=0}T^{-l} \xi = \mathscr{B}$ and  $\mathscr{C} \subset \mathscr{B}$ a $T$-invariant sub sigma algebra. Then for $m$ almost every $x \in X$ we have,
\[ \lim_{n \rightarrow \infty}\frac{1}{n}I_m\left( \bigvee^{n-1}_{l=0} T^{-l}\xi \bigg|\mathscr{C}\right)(x) \rightarrow h_{m}(T|\mathscr{C}).\]
\end{lemma}
\begin{proof} See the proof of\cite[Lemma 9.3.1]{LYMED2}. The Lemma is a mild generalisation of \cite[Theorem 7, Chapter 2]{Parry} and may proven in the same way.
\end{proof}

Fix some finite subset $\D_{0}\subseteq \D$ such that,
\begin{enumerate}
 \vspace{2mm}
 \item[(i)] If $\D$ is tall then there exists some $j^1, j^2 \in \N$ with $(i^1_0, j^1), (i^2_0, j^1) \in \D_0  $, where $i^1_0, i^2_0\in \A$ are as in the definition of $T_n(\omega)$,
 \vspace{2mm} 
 \item[(ii)] If $\D$ is wide then $(i',j^1_0),(i',j^2_0) \in \D_0$, where $(i',j^1_0), (i',j^2_0)\in \D$ are as in the definition of $W_n(\omega)$.
 \vspace{2mm} 
 \end{enumerate}

We shall let $ \B_{\sigma^q}^{0}(\Sigma)$ denote the set of $\mu \in \B_{\sigma^q}^{*}(\Sigma)$ which satisfy, $\mu([\omega_1,\cdots, \omega_q])>0$ for all $(\omega_1, \cdots, \omega_q) \in \D_{0}^q$.

\begin{lemma}\label{LOTS of limits} Given $\nu \in \E_{\sigma^q}^{0}(\Sigma)$, the following convergences hold for $\nu$ almost every $\omega \in \Sigma$,
\begin{enumerate}
 \vspace{3mm}
 \item[(i)] $\lim_{n \rightarrow \infty} A_{nq}(\psi)(\tau)=\int A_q(\psi) d \nu\circ \pi^{-1}$,
 \vspace{3mm}
 \item[(ii)] $\lim_{n \rightarrow \infty} A_{nq}(\chi)(\omega)=\int A_q(\chi) d \nu$,
 \vspace{3mm}
 \item[(iii)] $\lim_{n \rightarrow \infty}A_{nq}(\varphi_k)(\omega) = \int A_{q}(\varphi_k) d\nu$ for all $k \leq m$, 
 \vspace{3mm}
 \item[(iv)] $\lim_{n \rightarrow \infty} -n^{-1}\log \nu \circ \pi^{-1}([\pi(\omega)_1\cdots \pi(\omega)_{nq}])= h_{\nu \circ \pi^{-1}}(\sigma_v^q)$,
 \vspace{3mm}
 \item[(v)] $\lim_{n \rightarrow \infty} -n^{-1} \log \nu^{\pi(\omega)}([\omega_1\cdots \omega_{nq}])= h_{\nu}(\sigma^q|\pi^{-1}\mathscr{A})$,
 \vspace{3mm}
 \item[(vi)] $\lim_{n \rightarrow \infty} n^{-1} T_{n}(\tau)=0,$ provided $\D$ is tall,
 \vspace{3mm}
 \item[(vii)] $\lim_{n \rightarrow \infty} n^{-1} W_{n}(\omega)=0,$ provided $\D$ is wide.
 \vspace{3mm}
  \end{enumerate}
 \end{lemma}
\begin{proof} Limits (i)-(iii) follow from Birkhoff's ergodic theorem. Indeed since $\sigma^q$ is ergodic with respect to $\nu$, $\sigma_v^q$ is ergodic with respect to $\nu \circ \pi^{-1}$.  

If we let $\xi_v$ denote the partition of $\Sigma_v$ into cylinder sets of length $q$ and $\mathcal{N}:=\left\lbrace \Sigma_v, \emptyset \right\rbrace$ denote the null sigma algebra, then for each $\tau \in \Sigma_v$ we have,
\begin{equation*}
I_{\nu \circ \pi^{-1}}\left( \bigvee^{n-1}_{l=0} \sigma^{-lq}\xi_v \bigg| \hspace{1mm}\mathcal{N} \right)(\tau)=\log \nu \circ \pi^{-1}([\tau_1\cdots \tau_{nq}]).
\end{equation*}
Thus (i) follows from Lemma \ref{Ergodic theorem of information theory}. Similarly if we let $\xi_h$ denote the partition of $\Sigma$ into cylinder sets of length $q$ then for each $\omega \in \Sigma$ we have,
\begin{equation*}
I_{\nu}\left( \bigvee^{n-1}_{l=0} \sigma^{-lq}\xi_h \bigg| \pi^{-1} \mathscr{A} \right)(\omega)=\log \nu^{\pi(\omega)}([\omega_1\cdots \omega_{nq}]),
\end{equation*}
so (ii) also follows from Lemma \ref{Ergodic theorem of information theory}. 

For each $(\eta_1, \cdots, \eta_q) \in \D_0^q$ we have, 
\begin{align*}
\lim_{n \rightarrow \infty} \frac{1}{n} \#\left\lbrace  l < n: (\omega_{lq+1}, \cdots, \omega_{lq+q}=(\eta_1, \cdots, \eta_q) \right\rbrace\\ =\nu([\eta_1, \cdots, \eta_q])>0.
\end{align*}
Limits (vi) and (vii) follow.
\end{proof}

\begin{lemma}\label{Egorov type stuff} Given $\nu \in \E_{\sigma^q}^{0}(\Sigma)$, supported on some compact set $K$, along with constants $\delta, \epsilon>0$ and $m \in \N$, there exists $N \in \N$ and $U \subseteq \Sigma_v$ with $\nu\circ \pi^{-1}(U)>1-\delta$ such that for all $\tau \in U$ and all $n\geq N$ we have,
\begin{enumerate}
 \vspace{2mm}
 \item[(i)] $\displaystyle \bigg|\frac{\log \nu \circ \pi^{-1}([\tau_1\cdots \tau_{nq}])}{S_{nq}(\psi)(\tau)}+\frac{ h_{\nu \circ \pi^{-1}}( \sigma_v^q)}{\int S_q(\psi) d \nu\circ \pi^{-1}}\bigg|<\epsilon$,
 \vspace{4mm}
 \item[(ii)] $\displaystyle T_{n}(\tau)< n \epsilon,$ provided $\D$ is tall.
 \vspace{2mm}
 \end{enumerate}
 Moreover, for each $\tau \in U$ there exists $V_{\tau} \subseteq \pi^{-1}\{\tau\} \cap K$ with $\nu^{\tau}(V_{\tau})>1-\delta$ and for all $\omega \in V_{\tau}$ and $n \geq N$ we have,
 \begin{enumerate}
 \vspace{2mm}
 \item[(iii)] $\displaystyle \bigg|\frac{\log \nu^{\tau}([\omega_1\cdots \omega_{nq}])}{S_{nq}(\chi)(\omega)}+\frac{ h_{\nu}(\sigma^q|\pi^{-1}\mathscr{A})}{\int S_q(\chi) d \nu}\bigg|<\epsilon$,
 \vspace{4mm}
 \item[(iv)] $\displaystyle W_{n}(\omega)< n \epsilon,$ provided $\D$ is wide,
 \vspace{4mm}
 \item[(v)] $\displaystyle \bigg| A_{nq}(\varphi_k)(\omega) - \int A_{q}(\varphi_k) d\nu \bigg| < \epsilon$ for all $k \leq m$.
 \vspace{2mm}
  \end{enumerate}
 \end{lemma}

\begin{proof}
By Lemma \ref{LOTS of limits} (i), (iv) and (vi) combined with Egorov's theorem, there exists a set $U'' \subset \Sigma_v$ with $\nu\circ \pi^{-1}(U'')>1-\delta/2$,
such that for all $\tau \in U''$ and all $n\geq N'$ we have,
\begin{enumerate}
 \vspace{2mm}
 \item[(i)] $\displaystyle \bigg|\frac{\log \nu \circ \pi^{-1}([\tau_1\cdots \tau_{nq}])}{S_{nq}(\psi)(\tau)}+\frac{ h_{\nu \circ \pi^{-1}}( \sigma_v^q)}{\int S_q(\psi) d \nu\circ \pi^{-1}}\bigg|<\epsilon$,
 \vspace{4mm}
 \item[(ii)] $\displaystyle T_{n}(\tau)< n \epsilon,$ provided $\D$ is tall.
 \vspace{2mm}
 \end{enumerate}By Lemma \ref{LOTS of limits} (ii), (iii), (v) and (vii) we may take $U' \subset U''$ with $\nu\circ \pi^{-1}(U')=\nu\circ \pi^{-1}(U'')$ such 
for all $\tau \in U'$, $\nu^{\tau}$ is supported on $\pi^{-1}\{\tau\} \cap K$ and for $\nu^{\tau}$ almost all $\omega \in \subseteq \pi^{-1}\{\tau\}\cap K$ we have,
 \begin{enumerate}
 \vspace{2mm}
 \item[(iii)'] $\displaystyle \lim_{n \rightarrow \infty}\frac{\log \nu^{\tau}([\omega_1\cdots \omega_{nq}])}{S_{nq}(\chi)(\omega)}=\frac{ h_{\nu}(\sigma^q|\pi^{-1}\mathscr{A})}{\int S_q(\chi) d \nu}$,
 \vspace{4mm}
 \item[(iv)'] $\displaystyle \lim_{n \rightarrow \infty} n^{-1} W_{n}(\omega)=0,$ provided $\D$ is wide,
 \vspace{4mm}
 \item[(iv)'] $\displaystyle \lim_{n \rightarrow \infty} A_{nq}(\varphi_k)(\omega) = \int A_{q}(\varphi_k) d\nu$ for all $k \leq m$.
 \vspace{4mm}
 \end{enumerate}
Applying Egorov's theorem once more, we obtain for each $\tau \in U'$ a set $V_{\tau} \subseteq \pi^{-1}\{\tau\}\cap K$ with $\nu^{\tau}(V_{\tau})>1-\delta$ such that for all $\omega \in V_{\tau}$ and $n \geq N'(\tau)$ we have,
 \begin{enumerate}
 \vspace{2mm}
 \item[(iii)] $\displaystyle \bigg|\frac{\log \nu^{\tau}([\omega_1\cdots \omega_{nq}])}{S_{nq}(\chi)(\omega)}+\frac{ h_{\nu}(\sigma^q|\pi^{-1}\mathscr{A})}{\int S_q(\chi) d \nu}\bigg|<\epsilon$,
 \vspace{4mm}
 \item[(iv)] $\displaystyle W_{n}(\omega)< n \epsilon,$ provided $\D$ is wide,
 \vspace{4mm}
 \item[(iv)] $\displaystyle \bigg| A_{nq}(\varphi_k)(\omega) - \int A_{q}(\varphi_k) d\nu \bigg| < \epsilon$ for all $k \leq m$.
 \vspace{2mm}
  \end{enumerate}
Now choose $U\subseteq U'$ with  $\nu\circ \pi^{-1}(U)>\nu\circ \pi^{-1}(U')-\delta/2>1-\delta$ for which,
\begin{equation*}
N:= \max \left\lbrace N(\tau) : \tau \in U \right\rbrace <\infty.
\end{equation*}
\end{proof}

\section{Proof of the lower bound}\label{Proof LB}

Throughout the proof of the lower bound we shall fix some $\alpha=(\alpha_k)_{k \in \N} \subset \R \cup \{\infty\}$. We define,
\begin{equation*}
\delta(\alpha):=\lim_{m \rightarrow \infty} \sup \left\lbrace D(\mu):\mu \in \M^*_{\sigma}(\Sigma),  \int \varphi_k d\mu \in B_m(\alpha_k) \text{ for } k \leq m \right\rbrace.
\end{equation*}
In this section we shall prove the following.
\begin{prop} $\dim J_{\varphi}(\alpha)\geq \delta(\alpha)$.
\end{prop}
Clearly we may assume that $\delta(\alpha)>-\infty$. Thus, by a simple compactness argument there exists $\delta_h(\alpha), \delta_v(\alpha) \in \R$ with $\delta_h(\alpha)+ \delta_v(\alpha)=\delta(\alpha)$, along with a sequence of measures $\{\mu_m\}_{m \in \N} \subset \M^*_{\sigma}(\Sigma)$ with 
 \begin{enumerate}
 \vspace{2mm}
  \item[A(i)] $\displaystyle  \frac{ h_{\mu_m \circ \pi^{-1}}( \sigma_v)}{\int \psi d \mu_m \circ \pi^{-1}} >\delta_v(\alpha)-\frac{1}{3m}$,
 \vspace{4mm}
 \item[A(ii)] $\displaystyle \frac{h_{\mu_m}(\sigma|\pi^{-1}\mathscr{A})}{\int \chi d \mu_m }>\delta_h(\alpha)-\frac{1}{3m}$,
 \vspace{4mm}
 \item[A(iii)] $\displaystyle \int \varphi_k d\mu_m \in B_{3m}(\alpha_k)$ for $k \leq m$.
  \vspace{2mm}
 \end{enumerate} 
Now choose $\delta_m>0$ for each $m \in \N$ in such a way that $\prod_{m=1}^{\infty}(1-\delta_m)>0$.

By Lemma \ref{exists approximate measure}, for each $m \in \N$, there exists ${q(m)} \geq m$ and $\nu_m \in \B_{\sigma^{q(m)}}^{0}(\Sigma)$ satisfying,
\begin{enumerate}
 \vspace{2mm}
 \item[B(i)] $\displaystyle \frac{ h_{\nu_m \circ \pi^{-1}}( \sigma_v^{q(m)})}{\int S_{q(m)}(\psi) d \nu_m \circ \pi^{-1}}> \delta_v(\alpha) -\frac{1}{2m}$,
 \vspace{4mm}
 \item[B(ii)] $\displaystyle \frac{ h_{\nu_m}(\sigma^{q(m)}|\pi^{-1}\mathscr{A})}{\int S_{q(m)}(\chi) d \nu_m}>\delta_h(\alpha) -\frac{1}{2m}$,
 \vspace{4mm}
 \item[B(iii)] $\displaystyle \int A_{{q(m)}}(\varphi_k) d\nu_m \in B_{2m}(\alpha)$ for all $k \leq m$,
\vspace{4mm}
 \item[B(iv)] $\displaystyle \var_{n}\left( A_{n}(\varphi_k)\right)< \frac{1}{m}$, for all $n\geq q(m)$ and $k \leq m$,
\vspace{4mm}
 \item[B(v)] $\displaystyle \var_{n}\left( A_{n}(\chi)\right), \var_{n}\left( A_{n}(\psi)\right)< \frac{1}{m}$, for all $n\geq q(m)$.
 \vspace{2mm}
  \end{enumerate}
Since $\nu_m \in \B_{\sigma^{q(m)}}^{0}(\Sigma)$ is compactly supported there is a finite digit set $\D_m \subset \D$ such that $\nu_m$ is supported on $\D_m^{\N}$. We define,
\begin{align*}
A(m):= \sup \left(\left\lbrace - \log \inf_{x \in [0,1]}|f_{d}|: d \in \D_m \right\rbrace \cup \bigg\{   |\var_1(\varphi_k)(\omega)|: \omega_1, \cdots, \omega_{q(m)} \in \D_m \bigg\} \right).
\end{align*}
Note that $A(m)$ is finite since $\D_m$ is finite and $\var_{1}(\varphi_k))$ is finite for all $k\leq m$.

For each $m$ we define,
$\tilde{A}(m):=\prod_{l=1}^{m+1}A(l) +1$. 

By Lemma \ref{Egorov type stuff}, for each $m \in \N$ we may take $N(m) \in \N$ and $U(m) \subseteq \Sigma_v$ with $\nu_m\circ \pi^{-1}(U(m))>1-\delta_m$ such that for all $\tau \in U(m)$ and all $n\geq N(m)$ we have,
\begin{enumerate}
 \vspace{2mm}
 \item[C(i)] $\displaystyle \frac{-\log \nu_m \circ \pi^{-1}([\tau_1\cdots \tau_{nq}])}{S_{nq}(\psi)(\tau)} >\delta_v(\alpha) -\frac{1}{m}$,
 \vspace{4mm}
 \item[C(ii)] $\displaystyle \frac{T_{n}(\tau)}{n}< \frac{1}{m},$ provided $\D$ is tall.
 \vspace{2mm}
 \end{enumerate}
 Moreover, for each $\tau \in U(m)$ there exists $V_{\tau}(m) \subseteq \pi^{-1}\{\tau\} \cap \D_m^{\N}$ with $\nu_m^{\tau}(V_{\tau}(m))>1-\delta_m$ and for all $\omega \in V_{\tau}(m)$ and $n \geq N(m)$ we have,
 \begin{enumerate}
 \vspace{2mm}
 \item[C(iii)] $\displaystyle \frac{-\log \nu_m^{\tau}([\omega_1\cdots \omega_{nq}])}{S_{nq}(\chi)(\omega)}> \delta_h(\alpha) -\frac{1}{m}$,
 \vspace{4mm}
 \item[C(iv)] $\displaystyle \tilde{A}(m)\frac{W_{n}(\omega)}{n}< \frac{1 }{m},$ provided $\D$ is wide,
 \vspace{4mm}
 \item[C(v)] $\displaystyle A_{nq}(\varphi_k)(\omega) \in B_m(\alpha_k)$ for all $k \leq m$.
 \vspace{2mm}
  \end{enumerate}
We now define a rapidly increasing sequence $(\gamma_m)_{m\in\N\cup\{0\}}$ of natural numbers by $\gamma_0=2N(1)$, $\gamma_1=2N(2)$ and for $m>1$ we let
\begin{equation}
\gamma_m:= (m+1)!\cdot \gamma_{m-1}  \left(\prod_{l=1}^{m+1}N(l)\right)\left(\prod_{l=1}^{m+1}A(l)\right)\left(\prod_{l=1}^{m+1}q(l)\right)+\gamma_{m-1}.
\end{equation}

We now define a measure $\W$ on $\Sigma_v$ by first defining $\W$ on a semi-algebra of cylinders and then extending $\W$ to a Borel probability measure on $\Sigma_v$ via the Daniell-Kolmogorov consistency theorem (\cite{Walters} Theorem 0.5). Given a cylinder  $[\tau_1\cdots\tau_{\gamma_M}]$ of length $\gamma_M$ for some $M \in \N$ we define
\begin{equation}
\W([\tau_1\cdots\tau_{\gamma_M}]):=\prod_{m=1}^M\nu_{m}\circ \pi^{-1}([\tau_{\gamma_{m-1}+1}\cdots \tau_{\gamma_m}]).
\end{equation}
Define $U\subseteq \Sigma_v$ by
\begin{equation}
U:=\bigcap_{m=1}^{\infty}\left\lbrace \tau\in \Sigma_v:[\tau_{\gamma_{m-1}+1}\cdots \tau_{\gamma_m}]\cap U(m)\neq \emptyset\right\rbrace .
\end{equation}
For each $\tau \in U$ and $m \in \N$ we choose $\hat{\tau}^{m} \in [\tau_{\gamma_{m-1}+1}\cdots \tau_{\gamma_m}]\cap U(m)$ and define a measure $\Z^{\tau}$ on $\Sigma$ by
\begin{equation}
\Z^{\tau}([\omega_1\cdots \omega_{\gamma_M}]):=\prod_{m=1}^M\nu_{m}^{\hat{\tau}^{m}}([\omega_{\gamma_{m-1}+1}\cdots \omega_{\gamma_m}]).
\end{equation}
\begin{lemma} For all $\tau \in U$ we have $\Z^{\tau}\left( \pi^{-1}\{\tau\}\right)=1$.
\end{lemma}
\begin{proof} For each $m\in \N$ we have,
\begin{eqnarray}
\nu_{m}^{\hat{\tau}^{m}}\left(\pi^{-1}[\tau_{\gamma_{m-1}+1}\cdots \tau_{\gamma_m}]\right)=\nu_{m}^{\hat{\tau}^{m}}\left(\pi^{-1}\{\hat{\tau}^{m}\}\right)=1.
\end{eqnarray}
Hence, for each $M \in \N$ we have
\begin{equation}
\Z^{\tau}(\pi^{-1}[\tau_1\cdots \tau_{\gamma_M}])=\prod_{m=1}^M\nu_{m}^{\hat{\tau}^{m}}\left(\pi^{-1}[\tau_{\gamma_{m-1}+1}\cdots \tau_{\gamma_m}]\right)=1.
\end{equation}
The lemma follows.
\end{proof}
For each $\tau \in U$ we define $V_{\tau}\subseteq \pi^{-1}\{\tau\}$ by
\begin{equation}
V_{\tau}:=\bigcap_{m=1}^{\infty}\left\lbrace \omega\in \Sigma:[\omega_{\gamma_{m-1}+1}\cdots \omega_{\gamma_m}]\cap V_{\tau}(m)\neq \emptyset\right\rbrace .
\end{equation}
We also define,
\begin{equation}
S:=\left\lbrace \omega \in \Sigma: \pi(\omega) \in U \text{ and } \omega \in V_{\pi(\omega)}\right\rbrace.
\end{equation}
We shall show that $S\subseteq E_{\varphi}(\alpha)$ and $\dim \Pi(S)\geq \delta(\alpha)$.

\begin{lemma}\label{S subset} $S\subseteq E_{\varphi}(\alpha)$.
\end{lemma}
\begin{proof} Note that it suffices to take $\omega \in \Sigma$ with $\pi(\omega) \in U$ and $\omega \in V_{\pi(\omega)}$ and show that for each $k \in \N$ we have $\lim_{n \rightarrow \infty}A_n(\varphi_k)(\omega)= \alpha_k$. 

Given $n \in \N$ we choose $m(n) \in \N$ so that $\gamma_{m(n)} \leq n < \gamma_{m(n)+1}$ and $l(n) \in \N$ so that  $\gamma_{m(n)}+l(n)q(m(n)+1)\leq n < \gamma_{m(n)}+(l(n)+1)q(m(n)+1).$ Since $\lim_{n \rightarrow \infty} m(n)= \infty$ we may choose $n(k)\in \N$ so that $m(n) \geq k$ for all $n \geq n(k)$. 

First lets suppose that $\alpha_k \in \R$. Given $n\geq n(k)$, either $l(n) \leq N(m(n)+1)$, in which case
\[\big| S_n(\varphi_k)(\omega)-n\alpha_k\big|\]
\begin{eqnarray*}
&\leq & \big| S_{\gamma_{m(n)}-\gamma_{m(n)-1}}(\varphi_k)(\sigma^{\gamma_{m(n)-1}}\omega)-(\gamma_{m(n)}-\gamma_{m(n)-1})\alpha_k\big|\\  &&+   \gamma_{m(n)-1}A(m(n)-1)+(N(m(n)+1)+1)q(m(n)+1)A(m(n)+1)\\
 &\leq & \big| S_{\gamma_{m(n)}-\gamma_{m(n)-1}}(\varphi_k)(\sigma^{\gamma_{m(n)-1}}\omega)-(\gamma_{m(n)}-\gamma_{m(n)-1})\alpha_k\big| + \frac{2\gamma_{m(n)}}{m(n)}\\
&\leq & 
 \big| S_{\gamma_{m(n)}-\gamma_{m(n)-1}}(\varphi_k)(\hat{\omega}^{m(n)})-(\gamma_{m(n)}-\gamma_{m(n)-1})\alpha_k\big|\\
&&+\left(\gamma_{m(n)}-\gamma_{m(n)-1}\right)\var_{\gamma_{m(n)}-\gamma_{m(n)-1}}\left(A_{\gamma_{m(n)}-\gamma_{m(n)-1}}(\varphi_k)\right)+\frac{2n}{m(n)}\\
&\leq& \frac{2(\gamma_{m(n)}-\gamma_{m(n)-1})}{m(n)}+\frac{2n}{m(n)}\leq \frac{4n}{m(n)}.
\end{eqnarray*}
On the other hand, if $l(n) > N(m(n)+1)$ then we have,
\[ \big| S_n(\varphi_k)(\omega)-n\alpha_k\big| \]
\begin{eqnarray*}
 &\leq & \big| S_{\gamma_{m(n)}-\gamma_{m(n)-1}}(\varphi_k)(\sigma^{\gamma_{m(n)-1}}\omega)-(\gamma_{m(n)}-\gamma_{m(n)-1})\alpha_k\big|\\  
 &&+  \big| S_{l(n)q(m(n)+1)-\gamma_{m(n)}}(\varphi_k)(\sigma^{\gamma_{m(n)}}\omega)-(l(n)q(m(n)+1)-\gamma_{m(n)})\alpha_k\big|\\ 
 && + \gamma_{m(n)-1}A(m(n)-1)+q(m(n)+1)A(m(n)+1)\\
 &\leq &  \big| S_{\gamma_{m(n)}-\gamma_{m(n)-1}}(\varphi_k)(\sigma^{\gamma_{m(n)-1}}\omega)-(\gamma_{m(n)}-\gamma_{m(n)-1})\alpha_k\big|\\  
 &&+  \big| S_{l(n)q(m(n)+1)-\gamma_{m(n)}}(\varphi_k)(\sigma^{\gamma_{m(n)}}\omega)-(l(n)q(m(n)+1)-\gamma_{m(n)})\alpha_k \big| +  \frac{2\gamma_{m(n)}}{m(n)} \\
&\leq &  \big| S_{\gamma_{m(n)}-\gamma_{m(n)-1}}(\varphi_k)(\hat{\omega}^{m(n)})-(\gamma_{m(n)}-\gamma_{m(n)-1})\alpha_k\big|\\ &&+  \big| S_{l(n)q(m(n)+1)-\gamma_{m(n)}}(\varphi_k)(\hat{\omega}^{m(n)+1})-(l(n)q(m(n)+1)-\gamma_{m(n)})\alpha_k\big|\\
&& +  (\gamma_{m(n)}-\gamma_{m(n)-1})\var_{\gamma_{m(n)}-\gamma_{m(n)-1}}A_{\gamma_{m(n)}-\gamma_{m(n)-1}}(\varphi_k) \\ 
&&+ (l(n)q(m(n)+1)-\gamma_{m(n)})\var_{l(n)q(m(n)+1)-\gamma_{m(n)}}A_{l(n)q(m(n)+1)-\gamma_{m(n)}}(\varphi_k)+ \frac{2\gamma_{m(n)}}{m(n)}\\
&\leq &  \frac{2(\gamma_{m(n)}-\gamma_{m(n)-1})}{m(n)}+\frac{2(l(n)q(m(n)+1)-\gamma_{m(n)})}{m(n)}+ \frac{2\gamma_{m(n)}}{m(n)}\leq \frac{6 n}{m(n)}.
\end{eqnarray*}
Thus, for all $n \geq n(k)$ we have,
\begin{equation*}
\big| A_n(\varphi_k)(\omega)-\alpha_k\big|\leq  \frac{6}{m(n)}.
\end{equation*}
Since $\lim_{n \rightarrow \infty} m(n)=\infty$ the lemma holds when $\alpha_k$ is finite.

Now suppose that $\alpha_k = \infty$. Given $n\geq n(k)$, either $l(n) \leq N(m(n)+1)$, in which case,
\begin{eqnarray*}
 S_n(\varphi_k)(\omega) &\geq & S_{\gamma_{m(n)}-\gamma_{m(n)-1}}(\varphi_k)(\sigma^{\gamma_{m(n)-1}}\omega)\\
 & \geq &  S_{\gamma_{m(n)}-\gamma_{m(n)-1}}(\varphi_k)(\hat{\omega^{m(n)}}) \\ && -\left(\gamma_{m(n)}-\gamma_{m(n)-1}\right)\var_{\gamma_{m(n)}-\gamma_{m(n)-1}}\left(A_{\gamma_{m(n)}-\gamma_{m(n)-1}}(\varphi_k)\right)\\
 & \geq & (\gamma_{m(n)}-\gamma_{m(n)-1})m(n)- \frac{\gamma_{m(n)}-\gamma_{m(n)-1}}{m(n)}\\
 & \geq & n m(n) - (\gamma_{m(n)-1}+l(n)q(m(n)+1))m(n)- \frac{n}{m(n)}\\
 &\geq & n m(n) - \frac{2 \gamma_{m(n)}}{m(n)}- \frac{n}{m(n)}\\
  &\geq & n m(n) - \frac{3n}{m(n)}.
  \end{eqnarray*}
On the other hand, if $l(n) > N(m(n)+1)$ then we have,
\begin{eqnarray*}
S_n(\varphi_k)(\omega) &\geq & S_{\gamma_{m(n)}-\gamma_{m(n)-1}}(\varphi_k)(\sigma^{\gamma_{m(n)-1}}\omega) + S_{l(n)q(m(n)+1)-\gamma_{m(n)}}(\varphi_k)(\sigma^{\gamma_{m(n)}}\omega)\\
&\geq & S_{\gamma_{m(n)}-\gamma_{m(n)-1}}(\varphi_k)(\hat{\omega}^{m(n)})+S_{l(n)q(m(n)+1)-\gamma_{m(n)}}(\varphi_k)(\hat{\omega}^{m(n)+1}) 
\\ & & - (\gamma_{m(n)}-\gamma_{m(n)-1})\var_{\gamma_{m(n)}-\gamma_{m(n)-1}}A_{\gamma_{m(n)}-\gamma_{m(n)-1}}(\varphi_k) \\ &&- (l(n)q(m(n)+1)-\gamma_{m(n)})\var_{l(n)q(m(n)+1)-\gamma_{m(n)}}A_{l(n)q(m(n)+1)-\gamma_{m(n)}}(\varphi_k)\\
& \geq & \left(\gamma_{m(n)}-\gamma_{m(n)-1}\right)m(n) + \left(l(n)q(m(n)+1)-\gamma_{m(n)}\right)m(n)- \frac{2n}{m(n)}\\
& \geq & n m(n) - \frac{4n}{m(n)}.
\end{eqnarray*}  
Thus, for all $n \geq n(k)$ we have,
\begin{equation*}
A_n(\varphi_k)(\omega) \geq m(n) - \frac{4}{m(n)}.
\end{equation*}  
Letting $n\rightarrow \infty$ proves the lemma. 
\end{proof}

\begin{lemma}\label{+measure}\label{U>0}
$\W(U)>0$ and for each $\tau \in U$, $\Z^{\tau}\left(V_{\tau}\right)>0$.
\end{lemma}
\begin{proof}
\begin{equation*}
\W(U)\geq \prod_{m=1}^{\infty}\nu_m \circ \pi^{-1}(U(m))>\prod_{m=1}^{\infty}(1-\delta_m)>0.
\end{equation*}
Similarly for each $\tau \in U$ we have,
\begin{equation*}
\Z^{\tau}(V_{\tau}) \geq \prod_{m=1}^{\infty}\nu^{\hat{\tau}^m}_m (V_{\tau}(m))>\prod_{m=1}^{\infty}(1-\delta_m)>0.
\end{equation*}
\end{proof}

\begin{lemma}\label{Symbdimest} For all $\tau \in U$ and all $\omega \in V_{\tau}$ we have,
\begin{enumerate}
\vspace{4mm}
\item [(i)] $\displaystyle \liminf_{n\rightarrow \infty}\frac{-\log \W([\tau_1\cdots \tau_{n}])}{S_{n}(\psi)(\tau)} \geq \delta_v(\alpha)$,
\vspace{4mm}
\item[(ii)] $\displaystyle  \liminf_{n\rightarrow \infty} \frac{-\log \Z^{\tau}([\omega_1\cdots \omega_{n}])}{S_{n}(\chi)(\omega)}\geq  \delta_h(\alpha)$.
\end{enumerate}
\end{lemma}
\begin{proof}
We prove (ii). The proof of (i) is similar.
Take $\tau \in U$ and $\omega \in V_{\tau}$. Given $n \in \N$ we choose $m(n) \in \N$ so that $\gamma_{m(n)} \leq n < \gamma_{m(n)+1}$ and $l(n) \in \N$ so that  $\gamma_{m(n)}+l(n)q(m(n)+1)\leq n < \gamma_{m(n)}+(l(n)+1)q(m(n)+1).$ If $l(n) \leq N(m(n)+1)$ then  by C(iii) we have,
\begin{eqnarray*}
-\log \Z^{\tau}([\omega_1\cdots \omega_{n}])& \geq & -\log \nu^{\hat{\tau}^m}_m \left([\omega_{\gamma_{m(n)-1}+1} \cdots \omega_{\gamma_{m(n)}}]\right)\\
& \geq & S_{\gamma_{m(n)}-\gamma_{m(n)-1}}(\chi)(\hat{\omega}^{m(n)})\left(\delta_{h}(\alpha)- \frac{1}{m(n)}\right)
\end{eqnarray*}
where 
\begin{eqnarray*}
\hat{\omega}^{m(n)} \in [\omega_{\gamma_{m(n)-1}+1} \cdots \omega_{\gamma_{m(n)}}]\cap V_{\pi(\omega)}(m).
\end{eqnarray*} 

Moreover, using B(v) combined with $\gamma_{m(n)}\leq n$ we have,
\[S_{\gamma_{m(n)}-\gamma_{m(n)-1}}(\chi)(\hat{\omega}^{m(n)})\]
\begin{eqnarray*}
& \geq & S_{\gamma_{m(n)}-\gamma_{m(n)-1}}(\chi)(\sigma^{\gamma_{m(n)-1}}\omega) - (\gamma_{m(n)}-\gamma_{m(n)-1}) \var_{\gamma_{m(n)}-\gamma_{m(n)-1}}  A_{\gamma_{m(n)}-\gamma_{m(n)-1}}(\chi) \\
&\geq& S_{\gamma_{m(n)}-\gamma_{m(n)-1}}(\chi)(\sigma^{\gamma_{m(n)-1}}\omega) - \frac{\gamma_{m(n)}-\gamma_{m(n)-1}}{m(n)}\\
& \geq & S_n(\chi)(\omega) - \gamma_{m(n)-1}A(m(n)-1)-(N(m(n)+1)+1)q(m(n)+1)A(m(n)+1)   - \frac{n}{m(n)}\\
& \geq & S_n(\chi)(\omega) - \frac{2\gamma_{m(n)}}{m(n)}   - \frac{n}{m(n)}\\
& \geq & S_n(\chi)(\omega) - \frac{3n}{m(n)}.
\end{eqnarray*}

On the other hand, if $l(n)>N(m(n)+1)$ then by C(iii) we have
\[-\log \Z^{\tau}([\omega_1\cdots \omega_{n}])\]
\begin{eqnarray*}
& \geq & -\log \nu^{\hat{\tau}^m}_m \left([\omega_{\gamma_{m(n)-1}+1} \cdots \omega_{\gamma_{m(n)}}]\right)-\log \nu^{\hat{\tau}^m}_m \left([\omega_{\gamma_{m(n)}+1} \cdots \omega_{\gamma_{m(n)}+l(n)q(m(n)+1)}]\right) ,\\
& \geq & \left(S_{\gamma_{m(n)}-\gamma_{m(n)-1}}(\chi)(\hat{\omega}^{m(n)})+ S_{l(n)q(m(n)+1)}(\chi)(\hat{\omega}^{m(n)+1}) \right)\left(\delta_{h}(\alpha)- \frac{1}{m(n)}\right),
\end{eqnarray*}
\begin{eqnarray*}
\hat{\omega}^{m(n)} &\in& [\omega_{\gamma_{m(n)-1}+1} \cdots \omega_{\gamma_{m(n)}}]\cap V_{\pi(\omega)}(m(n)),\\
\hat{\omega}^{m(n)+1} &\in & [\omega_{\gamma_{m(n)}+1} \cdots \omega_{\gamma_{m(n+1)}}]\cap V_{\pi(\omega)}(m(n)+1).
\end{eqnarray*} 

As before, using B(v) combined with $\gamma_{m(n)}+l(n)q(m(n)+1)\leq n$ we have,
\begin{eqnarray*}
S_{\gamma_{m(n)}-\gamma_{m(n)-1}}(\chi)(\hat{\omega}^{m(n)}) & \geq & S_{\gamma_{m(n)}-\gamma_{m(n)-1}}(\chi)(\sigma^{\gamma_{m(n)-1}}\omega) - \frac{n}{m(n)}\\
S_{l(n)q(m(n)+1)}(\chi)(\hat{\omega}^{m(n)+1}) & \geq & S_{l(n)q(m(n)+1)}(\chi)(\sigma^{\gamma_{m(n)}}\omega) - \frac{n}{m(n)}.
\end{eqnarray*}
It follows that,
\[S_{\gamma_{m(n)}-\gamma_{m(n)-1}}(\chi)(\hat{\omega}^{m(n)})+ S_{l(n)q(m(n)+1)}(\chi)(\hat{\omega}^{m(n)+1})\]
\begin{eqnarray*}
 & \geq & S_{\gamma_{m(n)}-\gamma_{m(n)-1}}(\chi)(\sigma^{\gamma_{m(n)-1}}\omega)+S_{l(n)q(m(n)+1)}(\chi)(\sigma^{\gamma_{m(n)}}\omega)- \frac{2n}{m(n)}\\
& \geq & S_n(\chi)(\omega)-\gamma_{m(n)-1}A(m(n)-1)-q(m(n)+1)A(m(n)+1)-\frac{2n}{m(n)}\\
&\geq & S_n(\chi)(\omega)-\frac{4n}{m(n)}.
\end{eqnarray*}
Thus, for all $n \in \N$ we have,
\begin{eqnarray*}
-\log \Z^{\tau}([\omega_1\cdots \omega_{n}])\geq \left(S_n(\chi)(\omega)-\frac{4n}{m(n)}\right)\left(\delta_{h}(\alpha)- \frac{1}{m(n)}\right).
\end{eqnarray*}
Since $\liminf_{n\rightarrow \infty}n^{-1}S_n(\chi)(\omega) \geq \xi >0$ and $\lim_{n \rightarrow \infty} m(n)=\infty$, the lemma holds.
\end{proof}

\begin{lemma}\label{No borders lemma} For all $\tau \in U$ and all $\omega \in V_{\tau}$ we have,
\begin{enumerate}
\vspace{2mm}
\item [(i)] $\lim_{n\rightarrow \infty} n^{-1}R_n(\omega)=0$, provided $\D$ is wide,
\vspace{4mm}
\item [(ii)] $\lim_{n \rightarrow \infty}n^{-1}R_n(\omega) W_n(\omega)=0$, provided $\D$ is wide,
\vspace{4mm}
\item [(iii)] $\lim_{n\rightarrow \infty} n^{-1}T_n(\tau)=0$, provided $\D$ is tall.
\vspace{2mm}
\end{enumerate}
\end{lemma}
\begin{proof}
We shall prove (i) and (ii) simultaneously. The proof of (iii) is similar to that of (ii). 
 
Suppose that $\D$ is wide and take $\omega \in V_{\tau}$ with $\tau \in U$. Given $n \in \N$ we choose $m(n)$ to be the maximal natural number with $\gamma_{m(n)-1}+N(m(n)) \leq n$. Now suppose $n(1+(\tilde{A}(m(n))m(n))^{-1})< \gamma_{m(n)}$. Then we may choose $\hat{\omega}^{m(n)} \in [\omega_{\gamma_{m(n)-1}+1}\cdots \omega_{\gamma_{m(n)}}]\cap V_{\tau}(m(n))$. It follows from C(iv) that,
\begin{eqnarray}
W_{n- \gamma_{m(n)-1}}(\hat{\omega}^{m(n)})&<& \frac{n- \gamma_{m(n)-1}}{\tilde{A}(m(n))m(n)}\\
&<&\frac{n}{\tilde{A}(m(n))m(n)}\\
&\leq &\gamma_{m(n)}-n.
\end{eqnarray}
Thus, there is some $l \in \left\lbrace n+1, \cdots, n(1+(\tilde{A}(m(n))m(n))^{-1})\right\rbrace$ with $\omega_{l}=(i',j^1_0)$ and $\omega_{l+1}=(i',j^2_0)$. It follows that $n+W_n(\omega) \leq \gamma_{m(n)}$ and hence, 
\begin{eqnarray*}
R_n(\omega) \leq \max_{l \leq m(n)} A(l) \leq \tilde{A}(m(n))< \frac{2 \gamma_{m(n)-1}}{m(n)-1}<\frac{2n}{m(n)}.
\end{eqnarray*}
Also, 
\begin{eqnarray*}
R_n(\omega) W_n(\omega)  \leq \tilde{A}(m(n)) \frac{n}{\tilde{A}(m(n))m(n)} < \frac{n}{m(n)}.
\end{eqnarray*}
On the other hand, suppose that $n(1+(\tilde{A}(m(n))m(n))^{-1})\geq \gamma_{m(n)}$. Then we may choose $\hat{\omega}^{m(n)+1} \in [\omega_{\gamma_{m(n)}+1}\cdots \omega_{\gamma_{m(n)+1}}]\cap V_{\tau}(m(n))$. By C(iv) we have,
\begin{eqnarray*}
W_{N(m(n)+1)}(\hat{\omega}^{m(n)+1})&<& \frac{N(m(n)+1)}{\tilde{A}(m(n)+1)(m(n)+1)}\\
&\leq &\gamma_{m(n)+1}-\gamma_{m(n)+1}.
\end{eqnarray*}
Thus, there is some $l \in \left\lbrace n+1, \cdots,\gamma_{m(n)}+N(m(n)+1)\left(\tilde{A}(m(n)+1)(m(n)+1)\right)^{-1}\right\rbrace$. It follows that,
\begin{eqnarray*}
n+ \W_n(\omega) \leq \gamma_{m(n)+1}+N(m(n)+1)\left(\tilde{A}(m(n)+1)(m(n)+1)\right)^{-1} <\gamma_{m(n)+2}.
\end{eqnarray*}
Hence,
\begin{eqnarray*}
R_n(\omega) &\leq& \max_{l \leq m(n)+1} A(l) \leq \tilde{A}(m(n)+1)\\
&<& \frac{2 \gamma_{m(n)}}{m(n)}< \frac{4n}{m(n)}.
\end{eqnarray*}
In addition, we have,
\begin{eqnarray*}
W_n(\omega) \leq \gamma_{m(n)}+\frac{N(m(n)+1)}{\tilde{A}(m(n)+1)(m(n)+1)}-n \leq \frac{n+N(m(n)+1)}{\tilde{A}(m(n))m(n)}.
\end{eqnarray*}
Hence, 
\begin{eqnarray*}
R_n(\omega) W_n(\omega)  \leq \frac{n+N(m(n)+1)}{m(n)} \leq \frac{n+\gamma_{m(n)}}{m(n)} \leq \frac{3n}{m(n)}.
\end{eqnarray*}
Thus, for all $n \in \N$ we have,
\begin{eqnarray*}
\max\left\lbrace \frac{R_n(\omega)}{n}, \frac{R_n(\omega) W_n(\omega)}{n} \right\rbrace \leq \frac{4}{m(n)}.
\end{eqnarray*}
Letting $n \rightarrow \infty$, and hence $m(n) \rightarrow \infty$, proves the lemma.
\end{proof}

To complete the proof of the lower bound we require a version of Marstrand's slicing lemma. 
\begin{lemma}\label{Slicing Lemma}
Let $J$ be any subset of $\R^2$, and let $K$ be any subset of the $y$-axis. If $\dim J \cap \left( \R \times \{y\}\right) \geq t$ for all $y \in K$, then $\dim J \geq t + \dim K$.
\end{lemma}
\begin{proof} See \cite[Corollary 7.12]{Falconer Fractal Geometry}. \end{proof}

\begin{lemma} \label{Lower bound S} $\dim \Pi(S) \geq \delta(\alpha)$.
\end{lemma}
\begin{proof}
Recall that, $S:=\left\lbrace \omega \in \Sigma: \pi(\omega) \in U \text{ and } \omega \in V_{\pi(\omega)}\right\rbrace.$
It follows that, $\Pi(S)= \bigcup_{\tau \in U} \Pi(V_{\tau})$. Thus, for each $y =\Pi_v(\tau) \in \Pi_v(U)$ with $\tau \in U$ we have 
$\Pi(S)\cap \left( \R \times \{y\}\right)= \Pi(V_{\tau})$, since $V_{\tau} \subseteq \pi^{-1}\{\tau\}$. Hence, by Lemma \ref{Slicing Lemma} suffices to prove that $\dim \Pi_v(U) \geq \delta_v(\alpha)$ and for each $\tau \in U$ we have $\dim \Pi(V_{\tau}) \geq \delta_h(\alpha)$. 

To see that $\dim \Pi_v(U) \geq \delta_v(\alpha)$ we consider two cases. Either $\delta_v(\alpha)=0$, in which case the supposition is trivial since $U \neq \emptyset$ by Lemma \ref{+measure}, or $\delta_v(\alpha)>0$. It follows from A(i) that for some $\mu \in \M_{\sigma}(\Sigma)$ we have $h_{\mu \circ \pi^{-1}}( \sigma_v)>0$. Consequently $\D$ must be tall. Thus, by Lemma \ref{No borders lemma} the hypotheses of Lemma \ref{SymbDimPDimV} are satisfied, and so by Lemma \ref{SymbDimPDimV} combined with Lemma \ref{Symbdimest} (i) for all $y=\Pi_v(\tau) \in \Pi(U)$ we have,
\begin{equation*}
\liminf_{r \rightarrow 0}\frac{\log \W \circ \Pi_v^{-1}( B(y,r))}{\log r}\geq \liminf_{n \rightarrow \infty}\frac{-\log \W([\tau|n])}{S_n(\psi)(\tau)}\geq \delta_v(\alpha).
\end{equation*}
Since, by Lemma \ref{+measure}, $\W\circ \Pi_v^{-1}(\Pi_v(U)) \geq \W(U)>0$, by Lemma \ref{dim lem} we have $\dim \Pi_v(U) \geq \delta_v(\alpha)$.

Now fix $\tau \in U$. To show that $\dim \Pi(V_{\tau}) \geq \delta_h(\alpha)$ we proceed similarly. If $\delta_h(\alpha)=0$ then by Lemma \ref{+measure} $V_{\tau} \neq \emptyset$ and so the supposition is trivial. If on the other hand $\delta_h(\alpha)>0$ then by A(ii) we have
$h_{\mu}(\sigma|\pi^{-1}\mathscr{A})>0$ for some $\mu \in \M_{\sigma}(\Sigma)$ and consequently $\D$ must be wide. Thus, by Lemma \ref{No borders lemma} the hypotheses of Lemma \ref{SymbDimPDimH} are satisfied, and so by Lemma \ref{SymbDimPDimH} combined with Lemma \ref{Symbdimest} (ii) for all $x=\Pi(\omega) \in \Pi(V_{\tau})$ we have,
\begin{equation*}
\liminf_{r \rightarrow 0}\frac{\log \Z^{\tau}\circ \Pi^{-1}( B(x,r))}{\log r}\geq \liminf_{n \rightarrow \infty}\frac{-\log \Z^{\tau}([\omega|n])}{S_n(\chi)(\omega)} \geq \delta_h(\alpha).
\end{equation*}
Again, by Lemma \ref{+measure}, $\Z^{\tau}\circ \Pi^{-1}(\Pi(V_{\tau})) \geq \W(V_{\tau})>0$, by Lemma \ref{dim lem} we have $\dim \Pi(V_{\tau}) \geq \delta_h(\alpha)$. Thus, by \ref{Slicing Lemma} the lemma holds.
\end{proof}

To complete the proof of the lower bound we note that by Lemma \ref{S subset} $\Pi(S) \subseteq J_{\varphi}(\alpha)$. Therefore, by Lemma \ref{Lower bound S} we have,
\begin{equation*}
\dim J_{\varphi}(\alpha) \geq \lim_{m \rightarrow \infty} \sup \left\lbrace D(\mu):\mu \in \M^*_{\sigma}(\Sigma),  \int \varphi_k d\mu \in B_m(\alpha_k) \text{ for } k \leq m \right\rbrace.
\end{equation*}

\end{document}